\documentclass[12pt]{amsart}
\usepackage{txfonts}
\usepackage{amssymb}
\usepackage{eucal}
\usepackage{graphicx}
\usepackage{amsmath}
\usepackage{amscd}
\usepackage[all]{xy}           

\usepackage{amsfonts,latexsym}
\usepackage{xspace}
\usepackage{epsfig}
\usepackage{float}
\usepackage{color}
\usepackage{fancybox}
\usepackage{colordvi}
\usepackage{multicol}
\usepackage{colordvi}
\usepackage{stmaryrd}

\newtheorem{theorem}{Theorem}[section]
\newtheorem{defn}[theorem]{Definition}
\newtheorem{lemma}[theorem]{Lemma}
\newtheorem{coro}[theorem]{Corollary}
\newtheorem{prop-def}{Proposition-Definition}[section]

\newcommand{\nc}{\newcommand}
\newcommand{\delete}[1]{}

\nc{\mlabel}[1]{\label{#1}}  
\nc{\mcite}[1]{\cite{#1}}  
\nc{\mref}[1]{\ref{#1}}  
\nc{\mbibitem}[1]{\bibitem{#1}} 

\delete{
\nc{\mlabel}[1]{\label{#1}  
{\hfill \hspace{1cm}{\bf{{\ }\hfill(#1)}}}}
\nc{\mcite}[1]{\cite{#1}{{\bf{{\ }(#1)}}}}  
\nc{\mref}[1]{\ref{#1}{{\bf{{\ }(#1)}}}}  
\nc{\mbibitem}[1]{\bibitem[\bf #1]{#1}} 
}

\nc{\bfk}{\mathbf{k}}
\nc{\Der}{\mathrm{Der}}
\nc{\Ker}{\mathrm{Ker}}

\begin{document}

\title{3-Lie-Rinehart Algebras }\footnotetext{ Corresponding author: Ruipu Bai, E-mail: bairuipu@hbu.edu.cn.}

\author{RuiPu  Bai}
\address{College of Mathematics and Information Science,
Hebei University
\\
Key Laboratory of Machine Learning and Computational\\ Intelligence of Hebei Province, Baoding 071002, P.R. China} \email{bairuipu@hbu.edu.cn}

\author{Xiaojuan Li}
\address{College of Mathematics and Information  Science,
Hebei University, Baoding 071002, China} \email{lixiaojuan1209@126.com}

\author{Yingli Wu}
\address{College of Mathematics and Information  Science,
Hebei University, Baoding 071002, China} \email{15733268503@163.com}

\date{}

\begin{abstract}

In  this paper, we define a class of  3-algebras which are called  3-Lie-Rinehart algebras. A 3-Lie-Rinehart algebra  is a triple $(L, A, \rho)$, where $A$ is a commutative associative algebra, $L$ is an $A$-module, $(A, \rho)$ is a 3-Lie algebra $L$-module and $\rho(L, L)\subseteq Der(A)$. We  discuss the basic structures, actions and crossed modules of 3-Lie-Rinehart algebras  and  construct 3-Lie-Rinehart algebras from given algebras, we also study the derivations  from  3-Lie-Rinehart algebras to 3-Lie $A$-algebras.
 From the study, we  see that there is much difference between 3-Lie algebras and  3-Lie-Rinehart algebras.

\end{abstract}

\subjclass[2010]{17B05, 17D99.}

\keywords{ 3-Lie algebra,  commutative associative algebra, 3-Lie Rinehart algebra,  action,  crossed module.}

\maketitle



\allowdisplaybreaks

\section{Introduction}

3-Lie algebras \cite{F}  are applied to the study of Nambu mechanics,  the supersymmetry and gauge symmetry transformations of the world-volume theory of multiple coincident
M2-branes \cite{BLG2,HHY,Ro1,T}. The study of 3-Lie algebras is getting more and more attention. In \cite{BCR,B,ST}, authors studied the symplectic structures on 3-Lie algebras,
and gave the close relation between 3-Lie algebras with 3-Pre-Lie algebras, 3-Lie bialgebras and 3-Lie Yang-Baxter equation. In \cite{B,B1,BGL}, authors constructed the tensor form of ske-symmetric solutions of
3-Lie Yang-Baxter equation in  3-Lie algebras, and in \cite{B2,B3,B4, B5}, authors constructed finite and infinite dimensional 3-Lie algebras by Lie algebras, commutative associative algebras, cubic matrices and linear functions,
and studied  a class of 3-Lie algebras (which are called canonical Nambu
3-Lie algebras), which are applied in  Nambu mechanics \cite{T}.  A canonical Nambu 3-Lie algebra is an infinite dimensional vector space of triple classical observables on a three-dimensional phase space with coordinates $x; y; z$, the multiplication is
given by
\begin{equation}\label{eq:nambu}
    [f_1, f_2, f_3]_\partial = \frac {\partial(f_1, f_2, f_3)}{\partial(x, y, z)}
\end{equation}
where the right-hand side is the Jacobian determinant of the vector function $(f_1, f_2, f_3)$. For the infinite dimensional canonical Nambu 3-Lie algebra $A_{\partial}=\sum\limits_{m\in  Z} F z\exp (mx) \oplus \sum\limits_{m\in Z}F y\exp (mx)$ \cite{B5},
 authors studied the structure of it, and proved that $A_{\partial}$ is simple, and the regular representation of $A_{\partial}$ is a Harish-Chandra module, and the inner derivation algebra gives rise to
intermediate series modules of the Witt algebras and contain the smallest full toroidal Lie algebras without center.

In this paper, we define a 3-Lie-Rinehart algebra $(L, A, \rho)$, where $L$ is a 3-Lie algebra, $A$ is a commutative algebra, $L$ is an $A$-module, $(A, \rho)$ is a 3-Lie algebra $L$-module and $\rho(L, L)\subseteq Der(A)$.
We  pay close attention  to the study of structures of $(L, A, \rho)$. We give the actions and modules of $(L, A, \rho)$ on a 3-Lie $A$-algebra $(R, A)$, crossed modules $(L, A, \beta, \partial)$ of
 $(L, A, \rho)$, and derivations from  $(L, A, \rho)$ to a 3-Lie $A$-algebra $(R, A)$.  From a given 3-Lie-Rinehart algebra  $(L, A, \rho)$, we construct 3-Lie-Rinehart algebras  $(L\otimes A, A, \rho_1)$ (Theorem \ref{thm:LtensorA}) and $(E, A, \bar{\rho})$ (Theorem \ref{thm:E}), and Lie-Rinehart algebras  $(L\wedge L, \rho_2)$ ( Theorem \ref{thm:LL}) and $(W(L, R, A), \rho_3)$ (Theorem \ref{thm:do}). We study the structures of derivations from $(L, A, \rho)$ to a 3-Lie $A$-algebra $(R, A)$,  and  prove that $(R,\beta)$ is an  $(L, A, \rho)$-module if and only if the semidirect product $(L\ltimes R, A, \rho_4)$ is a 3-Lie-Rinehart algebra, and any 3-Lie A-algebra homomorphism from $L$ to $R$ induces an action $\beta$ from $(L, A, \rho)$ on $(R, A)$.
 For any  3-Lie-Rinehart algebra central epimorphism $\partial$, we get the crossed module $ (L, A, \beta, \partial)$, in  addition, if $\rho$ satisfies $\rho(L, L)A(L)=0$, then we get a crossed module of Lie-Rinehart algebra $(W(L,L,A), \tilde{\rho_3})$ ( in Theorem \ref{thm:actioneta}).

{\bf Unless otherwise stated,  algebras and vector spaces are over a field $F$ of characteristic zero,  and $A$ denotes a commutative associative algebra over $F$.}

\vspace{2mm}For a vector space $V$ over $F$, $V$ is called {\bf an $A$-module} if there is a $F$-linear mapping
    $$\alpha: A\rightarrow gl(V)\quad ~~ \mbox{satisfying}\quad~~ \alpha(ab)=\alpha(a)\alpha(b), ~~ \alpha(a+b)=\alpha(a)+\alpha(b), ~~ \forall a, b\in A.$$
    \\ For $\forall a\in A, v\in V$, $\alpha(a)(v)$ is simply denoted by $av.$

{\bf A 3-Lie algebra} over a field $F$ is an $F$-vector space $L$ endowed with a linear multiplication $[\ ,\  ,\  ]: L\wedge L\wedge L\rightarrow L$ satisfying for all $x_{1},x_{2},x_{3},y_{2},y_{3}\in L,$
\begin{equation}\label{eq:jacobi}
[[x_{1},x_{2},x_{3}],y_{2},y_{3}]=[[x_{1},y_{2},y_{3}],x_{2},x_{3}]+[[x_{2},y_{2},y_{3}],x_{3},x_{1}]+[[x_{3},y_{2},y_{3}],x_{1},x_{2}].
\end{equation}

Let $L$ be a 3-Lie algebra,  $V$ be a vector space, and
$\rho: L\wedge L\rightarrow gl (V)$ be an $F$-linear mapping. If $\rho$ satisfies that for all $x_i\in L, 1\leq i\leq 4,$

\begin{equation}\label{eq:mod1}
[\rho(x_1, x_2), \rho(x_3, x_4)]=\rho([x_1, x_2, x_3], x_4)-\rho([x_1, x_2, x_4], x_3),
  \end{equation}

\begin{equation}
\label{eq:mod2} \rho([x_1, x_2, x_3], x_4)=\rho(x_1, x_2)\rho(x_3, x_4)+\rho(x_2, x_3)\rho(x_1, x_4)+\rho(x_3, x_1)\rho(x_2, x_4),
\end{equation}
\\
then $(V, \rho)$ is called {\bf a representation} of $L$, or $(V, \rho)$ is {\bf an $L$-module}.

We know that  $(V, \rho)$  is a 3-Lie algebra $L$-module if and only if $L\dot+V$ is a 3-Lie algebra  with the multiplication, for all $x_i\in L, v_i\in V$, $i=1, 2, 3, $
 \begin{equation}
\label{eq:semiproduct} [x_1+v_1, x_2+v_2, x_3+v_3]=[x_1, x_2, x_3]+\rho(x_1, x_2)v_3+\rho(x_2, x_3)v_1+\rho(x_3, x_1)v_2,
\end{equation}
which is called {\bf the semi-direct product 3-Lie algebra}, denoted by $L\ltimes_{\rho} V$ (or simply denoted by $L\ltimes V$). It is clear that $L$ is a subalgebra of $L\ltimes V$ and $V$ is an abelian ideal.

Thanks to Eq \eqref{eq:jacobi}, $(L, \mbox{ad})$ is a representation of the 3-Lie algebra $L$, and it is called  the regular representation of $L$, where
\begin{equation}\label{eq:ad}
\mbox{ad}: L\wedge L\rightarrow gl(L), ~~ x, y\in L,~~~ ad(x, y)z=[x, y, z], ~~~ \forall z\in L.
\end{equation}
And  $\mbox{ad}(L\wedge L)$ is a Lie algebra which is called {\bf the inner derivation algebra of the 3-Lie algebra $L$}.
Thanks to \eqref{eq:jacobi}, $\forall x_1, x_2, y_1, y_2\in L,$
$$[\mbox{ad}(x_1,  y_1), \mbox{ad}(x_2,  y_2)]=\mbox{ad}([x_1, y_1, x_2],  y_2)+\mbox{ad}(x_2, [x_1, y_1, y_2]).$$

\section{Basic structures of 3-Lie Rinehart algenras}
\mlabel{sec:rbl3rbl}

\begin{defn}\label{defin:rinehart}
Let $L$ be a 3-Lie algebra over $F$, $L$ be an $A$-module and $(A, \rho)$ be an $L$-module.
If   $\rho( L\wedge L)\subseteq Der(A)$, and
\begin{equation}\label{eq:Rinhart1}
    [x, y, az]=a[x, y, z]+\rho(x, y)a z,~~\forall x, y, z\in L, a\in A,
\end{equation}
  \begin{equation}\label{eq:Rinhart2}
 \rho(ax, y)=\rho(x,ay)=a\rho(x, y),~~\forall x, y\in L, a\in A,
 \end{equation}
then  $(L, A, \rho)$ is  called {\bf a 3-Lie-Rinehart algebra}. And  if $\rho=0$, then  $(L, A)$ is called {\bf a 3-Lie $A$-algebra.}
\end{defn}

Let $G$ be a Lie algebra over $F$, $G$ is an $A$-module and $(A, \rho)$ be  $G$-module.
If   $\rho(G)\subseteq Der(A)$, and
\begin{equation}\label{eq:RinhartL1}
    [x, az]=a[x,z]+\rho(x)a z,~ \rho(ax)=a\rho(x),~\forall x,z\in G, a\in A,
\end{equation}
  then  $(G, \rho)$ is a Lie-Rinehart algebra \cite{Cas}. And  if $\rho=0$, then  $G$ is called  a Lie $A$-algebra.

In the following, $Der(L)$ denotes the derivation algebra of  3-Lie algebra $L$,  and
  \begin{equation}\label{eq:Ader}
  Der_A(L): =\{ D~ | ~D\in Der(L), \mbox{and}~ D(ax)=aD(x), \forall a\in A, x\in L\}.
  \end{equation}

For convenience, for all $x_i\in L, b\in A, 1\leq i\leq 4,$  denote

\begin{equation}\label{eq:N}
\begin{split}
   N(x_{1},x_{2},x_{3},x_{4}) :=&\rho(x_{1},x_{2})\rho(x_{3},x_{4})+\rho(x_{3},x_{4})\rho(x_{1},x_{2})+\rho(x_{2},x_{3})\rho(x_{1},x_{4})\\
   &+\rho(x_{1},x_{4})\rho(x_{2},x_{3})+\rho(x_{3},x_{1})\rho(x_{2},x_{4})+\rho(x_{2},x_{4})\rho(x_{3},x_{1}).\\
   R(x_{1},x_{2},x_{3},x_{4}):=&(\rho(x_{1},x_{2})b)\rho(x_{3},x_{4})+(\rho(x_{3},x_{1})b)\rho(x_{2},x_{4})\\
   & +(\rho(x_{2},x_{3})b)\rho(x_{1},x_{4}).\\
\end{split}
  \end{equation}

\begin{theorem}
Let  $(L, A, \rho)$ be  a 3-Lie-Rinehart algebra over $A$. Then the following identities hold,  for all $a, a_i\in A, x_i\in L, 1\leq i\leq 5,$
\begin{equation}\label{eq:(1)}
\begin{split}
 & \rho(x_{2},x_{3})(a_{4}a_{5})[x_{1},x_{2},x_{3}]+\rho(x_{3},x_{1})(a_{4}a_{5})[x_{2},x_{4},x_{5}]
 +\rho(x_{1},x_{2})(a_{4}a_{5})[x_{3},x_{4},x_{5}]\\
 &+a_{4}(\rho(x_{1},x_{4})a_{5}[x_{5},x_{2},x_{3}] +\rho(x_{2},x_{4})a_{5}[x_{5},x_{3},x_{1}]+\rho(x_{3},x_{4})a_{5}[x_{5},x_{1},x_{2}])\\
 &- a_{5}(\rho(x_{1},x_{5})a_{4}[x_{4},x_{2},x_{3}]+\rho(x_{2},x_{5})a_{4}[x_{4},x_{3},x_{1}]+\rho(x_{3},x_{5})a_{4}[x_{4},x_{1},x_{2}])=0.\\  \end{split}
  \end{equation}

\begin{equation}\label{eq:(2)}
\begin{split}
  &  \rho(x_{2},x_{3})a[x_{1},x_{4},x_{5}]+\rho(x_{1},x_{4})a[x_{2},x_{3},x_{5}] +\rho(x_{3},x_{1})a[x_{2},x_{4},x_{5}]\\
  &+\rho(x_{2},x_{4})a[x_{3},x_{1},x_{5}]+\rho(x_{1},x_{2})a[x_{3},x_{4},x_{5}]+\rho(x_{3},x_{4})a[x_{1},x_{2},x_{5}]=0.\\
   \end{split}
  \end{equation}

\begin{equation}\label{eq:(3)}
 N(x_{1},x_{2},x_{3},x_{4})=0, ~~  R(x_{1},x_{2},x_{3},x_{4})=0.
  \end{equation}
\end{theorem}

\begin{proof} Apply Eqs \eqref{eq:jacobi}, \eqref{eq:mod1}, \eqref{eq:mod2}, \eqref{eq:Rinhart1},  \eqref{eq:Rinhart2} and \eqref{eq:N}, and a direct computation.
\end{proof}

From Definition \ref{defin:rinehart}, if $B$ is a subalgebra of the 3-Lie algebra $L$ satisfying $ax\in B$ for all $a\in A, x\in B$,
then $(B, A, \rho|_{B\wedge B})$ is a 3-Lie-Rinehart algebra, which is called {\bf the subalgebra of the 3-Lie-Rinehart algebra $(L, A, \rho)$}.

If $B$ is an ideal of the 3-Lie algebra $L$ and  satisfies $ax\in B$ for all $a\in A, x\in B$, then $(B, A, \rho|_{B\wedge B})$ is called {\bf a Hypo-ideal}  of the  3-Lie-Rinehart algebra $(L, A, \rho)$, in addition,
if $\rho|_{B\wedge L}=0$, then $(B, A)$ is a 3-Lie A-algebra, which is called {\bf an ideal of the 3-Lie-Rinehart algebra $(L, A, \rho)$}.

Denotes
\begin{equation*}
Z_L(A):= \{ a ~|~a\in A, ax=0, \forall x\in L\},~~  Z_{\rho}(L):= \{ x ~|~x\in L, [x, L, L]=0, \rho(x, L)=0\}.
\end{equation*}

\begin{theorem}\label{thm:prop1}
  Let $B$ and  $C$ be ideals of 3-Lie-Rinehart algebra $(L, A, \rho)$.  Then

 $ 1)$ $(L/B, A, \tilde{\rho})$ is a $3$-Lie-Rinehart algebra, where $\tilde{\rho}: (L/B)\wedge (L/B)\rightarrow Der(A)$ and  $\tilde{\rho}(x_1+B,x_2+B)=\rho(x_1,x_2)$, $a(x_1+B)=ax_1+B$, $\forall x_1, x_2\in L, a\in A$.

 $2)$ $B+C$ and $B\cap C$ are ideals of  $(L, A, \rho)$.

 $3)$ $ Z_{\rho}(L)$ is an ideal of the 3-Lie-Rinehart algebra $(L, A, \rho)$, which is called {\bf the center} of  $(L, A, \rho)$, and $Z_L(A)$ is an ideal of $A$.
\end{theorem}

\begin{proof}
  It is clear that quotient algebra $L/B$ is  an $A$-module. Since $\rho(B, L)=0$, $\tilde{\rho}$ is valid.

   For all $x_i\in L, b\in A, 1\leq x_i\leq 4,$
\begin{equation*}
\begin{split}
&[\tilde{\rho}(x_1+B,x_2+B),\tilde{\rho}(x_3+B,x_4+B)]=[\rho(x_1,x_2), \rho(x_3,x_4)]\\
=&\rho([x_1,x_2,x_3],x_4)+\rho(x_3,[x_1,x_2,x_4])\\
=&\tilde{\rho}([x_1,x_2,x_3]+B,x_4+B)+\tilde{\rho}(x_3+B,[x_1,x_2,x_4]+B).\\
\end{split}
\end{equation*}
  By a similar discussion to the above, we have
\begin{equation*}
\begin{split}
&\tilde{\rho}([x_1+B,x_2+B,x_3+B],x_4+B)=\tilde{\rho}(x_1+B,x_2+B)\tilde{\rho}(x_3+B,x_4+B)\\
&+\tilde{\rho}(x_2+B,x_3+B)\tilde{\rho}(x_1+B,x_4+B)+\tilde{\rho}(x_3+B,x_1+B)\tilde{\rho}(x_2+B,x_4+B).\\
\end{split}
\end{equation*}

Now for any $b\in A$,
\begin{equation*}
\begin{split}
&[x_1+B,x_2+B,b(x_3+B)]=[x_1+B,x_2+B, bx_3+B]\\
=&[x_1,x_2,bx_3]+B=(b[x_1,x_2,x_3]+B)+ (\rho(x_1, x_2)b x_3+B)\\
=&b [x_1+B,x_2+B,x_3+B]+\tilde{\rho}(x_1+B, x_2+B)b (x_3+B).\\
\end{split}
\end{equation*}
Therefore, $(L/B, A, \tilde{\rho})$ is a 3-Lie-Rinehart algebra. The results 2) and 3) follows from a direct computation.
\end{proof}

\begin{defn}\label{defin:homomorph}
  Let $(L, A, \rho)$ and $(L', A, \rho')$ be 3-Lie-Rinehart algebras, $f:L\rightarrow L'$ be a 3-Lie algebra homomorphism. If $f$ satisfies that
  \begin{equation*}
  f(ax)=af(x), ~~ \rho'(f(x),f(y))=\rho(x,y), ~~ \forall a\in A, x, y\in L,
  \end{equation*}
  then $f$ is called a {\bf 3-Lie-Rinehart algebra homomorphism}, furthermore,
  \\1) if $f$ is a bijection, then $f$ is called an isomorphism; \\2) if $f$ is epimorphism and $Ker(f)\subseteq Z(L)$, then $f$ is called {\bf a central epimorphism}.
  \end{defn}

\begin{theorem}\label{thm:prop2}  Let $(L, A, \rho)$ and $(L', A, \rho')$ be 3-Lie-Rinehart algebras, and $f:L\rightarrow L'$ be a 3-Lie-Rinehart algebra homomorphism.  Then we have the following conclusions:

  1) $Ker(f)=\{x~ |~ x\in L, f(x)=0\}$ is an ideal of 3-Lie-Rinehart algebra $(L, A, \rho)$.

  2) $(f(L), A, \rho'|_{f(L)\wedge f(L)})$ is a subalgebra of  $(L', A, \rho')$ which is isomorphic to the quotient $3$-Lie-Rinehart algebra $(L/Ker(f), A, \tilde{\rho})$.

  3) There is an one to one correspondence between the subalgebras of  $(L, A, \rho)$ containing $Ker(f)$ to  the  subalgebras of $(f(L), A, \rho'|_{f(L)\wedge f(L)})$, and ideals  correspond to ideals.
  \end{theorem}

\begin{proof} The result follows from a direct computation.
 \end{proof}

Now we construct some new 3-Lie-Rinehart algebras from a given 3-Lie-Rinehart algebra $(L, A, \rho)$.

\begin{theorem}\label{thm:LtensorA}
Let $(L, A, \rho)$ be a 3-Lie-Rinehart algebra and $B=A\otimes L=\{ax~|~a\in A, x\in L\}$. Then $(B, A, \rho_1)$ is  3-Lie-Rinehart algebra, where the multiplication is,  for all $x_1, x_2, x_3\in L$, $ a_1, a_2, a_3\in A,$

\begin{equation}\label{eq:tensor}
\begin{split}
   [a_{1}x_{1},a_{2}x_{2},a_{3}x_{3}]&=a_{1}a_{2}a_{3}[x_{1},x_{2},x_{3}]+a_{1}a_{2}\rho(x_{1},x_{2})a_{3}x_{3}+a_{2}a_{3}\rho(x_{2},x_{3})a_{1}x_{1}\\
   &+a_{3}a_{1}\rho(x_{3},x_{1})a_{2}x_{2},\\
\end{split}
  \end{equation}
and $\rho_1:(A\otimes L)\wedge(A\otimes L)\rightarrow Der(A)$ is defined by ~~~
   $ \rho_1(a_{1}x_{1},a_{2}x_{2})=a_{1}a_{2}\rho(x_{1},x_{2}).
$
\end{theorem}

\begin{proof}
 By Eqs \eqref{eq:(1)}, \eqref{eq:(2)} and \eqref{eq:(3)}, $B$ is a $3$-Lie algebra in the multiplication \eqref{eq:tensor}, and $B$ is an $A$-module. Thanks to Eqs \eqref{eq:mod1} and \eqref{eq:mod2}, for all $x_1, x_2, x_3\in L$ and $b, a_1, a_2, a_3\in A,$

\begin{equation*}
\begin{split}
&[\rho_1(a_1x_1,a_2x_2),\rho_1(a_3x_3,a_4x_4)]=[a_1a_2\rho(x_1,x_2),a_3a_4\rho(x_3,x_4)]\\
=&a_1a_2\rho(x_1,x_2)(a_3a_4\rho(x_3,x_4))-a_3a_4\rho(x_3,x_4)(a_1a_2\rho(x_1,x_2))\\
=&a_1a_2a_3\rho(x_1,x_2)a_4\rho(x_3,x_4)+a_1a_2a_4\rho(x_1,x_2)a_3\rho(x_3,x_4)\\
&+a_1a_2a_3a_4\rho(x_1,x_2)\rho(x_3,x_4)-a_1a_3a_4\rho(x_3,x_4)a_2\rho(x_1,x_2)\\
&-a_2a_3a_4\rho(x_3,x_4)a_1\rho(x_1,x_2)-a_1a_2a_3a_4\rho(x_3,x_4)\rho(x_1,x_2)\\
=&\rho_1([a_1x_1,a_2x_2,a_3x_3],a_4x_4)+\rho_1(a_3x_3,[a_1x_1,a_2x_2,a_4x_4]).
\end{split}
\end{equation*}

\begin{equation*}
\begin{split}
&\rho_1(a_1x_1,a_2x_2)\rho_1(a_3x_3,a_4x_4)+\rho_1(a_2x_2,a_3x_3)\rho_1(a_1x_1,a_4x_4)\\
&+\rho_1(a_3x_3,a_1x_1)\rho_1(a_2x_2,a_4x_4)\\
=&a_1a_2\rho(x_1,x_2)(a_3a_4\rho(x_3,x_4))+a_2a_3\rho(x_2,x_3)(a_1a_4\rho(x_1,x_4))\\
&+a_1a_3\rho(x_3,x_1)(a_2a_4\rho(x_2,x_4))\\
=&a_1a_2a_3\rho(x_1,x_2)a_4\rho(x_3,x_4)+a_1a_2a_4\rho(x_1,x_2)a_3\rho(x_3,x_4)\\
&+a_1a_2a_3a_4\rho(x_1,x_2)\rho(x_3,x_4)+a_1a_2a_3\rho(x_2,x_3)a_4\rho(x_1,x_4)\\
&+a_2a_3a_4\rho(x_2,x_3)a_1\rho(x_1,x_4)+a_1a_2a_3a_4\rho(x_2,x_3)\rho(x_1,x_4)\\
&+a_1a_2a_3\rho(x_3,x_1)a_4\rho(x_2,x_4)+a_1a_3a_4\rho(x_3,x_1)a_2\rho(x_2,x_4)\\
&+a_1a_2a_3a_4\rho(x_3,x_1)\rho(x_2,x_4)\\
=&\rho_1([a_1x_1,a_2x_2,a_3x_3],a_4x_4).\\
\end{split}
\end{equation*}
It follows that $(A, \rho_1)$ is a 3-Lie algebra $B$-module.

By Definition \ref{defin:rinehart} and Eq \eqref{eq:tensor},
\begin{equation*}
\begin{split}
[a_1x_1,a_2x_2,b(a_3x_3)]=&a_1a_2a_3b[x_1,x_2,x_3]+a_1a_2\rho(x_1,x_2)(ba_3)x_3\\
&+a_2a_3b\rho(x_2,x_3)a_1x_1+a_1a_3b\rho(x_3,x_1)a_2x_2\\
=&b(a_1a_2a_3[x_1,x_2,x_3]+a_1a_2\rho(x_1,x_2)a_3x_3+a_2a_3\rho(x_2,x_3)a_1x_1\\
&+a_1a_3b\rho(x_3,x_1)a_2x_2)+a_1a_2\rho(x_1,x_2)b(a_3x_3)\\
=&b[a_1x_1,a_2x_2,a_3x_3]+\rho_1(a_1x_1,a_2x_2)b(a_3x_3).\\
\end{split}
\end{equation*}
Therefore, $(B,A,\rho_1)$ is a 3-Lie-Rinehart algebra.
 \end{proof}

\begin{theorem}\label{thm:E}
Let $(L, A, \rho)$ be a 3-Lie-Rinehart algebra and $E=\{(x,a)~|~x\in L, a\in A\}$. Then $(E, A, \bar{\rho})$ is a 3-Lie-Rinehart algebra, where $\forall a, b,c\in A, ~~ x, y,z\in L,~~ k\in  F,$
\begin{equation}\label{eq:rtimes}
k(x, a)=(kx, ka), ~~(x,a)+(y,b)=(x+y,a+b),~~
 a(y,b)=(ay,ab),
\end{equation}
\begin{equation}\label{eq:1rtimes}
    [(x,a),(y,b),(z,c)]=([x,y,z],\rho(x,y)c+\rho(y,z)a+\rho(z,x)b),
\end{equation}
\begin{equation}\label{eq:2rtimes}
\bar{\rho}: E\wedge E\rightarrow Der(A),~~
    \bar{\rho}((x,a),(y,b))=\rho(x,y).
\end{equation}
\end{theorem}

\begin{proof} Thanks to Eq \eqref{eq:rtimes}, $E$ is an $A$-module, and the 3-ary linear multiplication defined by Eq \eqref{eq:1rtimes} is skew-symmetric. For all $x_i\in L, a_i\in A$, $1\leq i\leq 5,$ since
\hspace{-2cm} \begin{equation*}
\begin{split}
&[[(x_1,a_1),(x_2,a_2),(x_3,a_3)],(x_4,a_4),(x_5,a_5)]\\
=&[([x_1,x_2,x_3],\rho(x_1,x_2)a_3+\rho(x_2,x_3)a_1+\rho(x_3,x_1)a_2),(x_4,a_4),(x_5,a_5)]\\
=&\big([[x_1,x_2,x_3],x_4,x_5],\rho([x_1,x_2,x_3],x_4)a_5+\rho(x_4,x_5)\rho(x_1,x_2)a_3\\
&+\rho(x_4,x_5)\rho(x_2,x_3)a_1+\rho(x_4,x_5)\rho(x_3,x_1)a_2+\rho(x_5,[x_1,x_2,x_3])a_4\big),\\
\end{split}
\end{equation*}

\begin{equation*}
\begin{split}
&[[(x_1,a_1),(x_4,a_4),(x_5,a_5)],(x_2,a_2),(x_3,a_3)]+[[(x_2,a_2),(x_4,a_4),(x_5,a_5)],(x_3,a_3),(x_1,a_1)]\\
&+[[(x_3,a_3),(x_4,a_4),(x_5,a_5)],(x_1,a_1),(x_2,a_2)]\\
=&\big([[x_1,x_4,x_5],x_2,x_3],\rho([x_1,x_4,x_5],x_2)a_3+\rho(x_2,x_3)(\rho(x_1,x_4)a_5+\rho(x_4,x_5)a_1\\
&+\rho(x_5,x_1)a_4)+\rho(x_3,[x_1,x_4,x_5])a_2\big)-\big([[x_2,x_4,x_5],x_1,x_3],\rho([x_2,x_4,x_5],x_1)a_3\\
&+\rho(x_3,x_1)(\rho(x_2,x_4)a_5+\rho(x_4,x_5)a_2+\rho(x_5,x_2)a_4)-\rho(x_3,[x_2,x_4,x_5])a_1\big)\\
&-\big([[x_3,x_4,x_5],x_2,x_1],\rho([x_3,x_4,x_5],x_2)a_1+\rho(x_1,x_2)(\rho(x_3,x_4)a_5+\rho(x_4,x_5)a_3\\
&+\rho(x_5,x_3)a_4)+\rho(x_1,[x_3,x_4,x_5])a_2\big)\\
=&[[(x_1,a_1),(x_2,a_2),(x_3,a_3)],(x_4,a_4),(x_5,a_5)],
\end{split}
\end{equation*}
 Eq \eqref{eq:jacobi} holds. Therefore,  $E$ is a 3-Lie algebra.

Now we prove that $(A, \bar{\rho})$ is a 3-Lie algebra $E$-module. By Eq \eqref{eq:2rtimes}, we have $\bar{\rho}(E\wedge E)\subseteq Der(A)$, and for all $x_i\in L$, $ a_i\in A$, $1\leq i\leq 4,$
\begin{equation*}
\begin{split}
&\bar{\rho}([(x_1,a_1),(x_2,a_2),(x_3,a_3)],(x_4,a_4))+\bar{\rho}((x_3,a_3),[(x_1,a_1),(x_2,a_2),(x_4,a_4)])\\
=&\bar{\rho}(([x_1,x_2,x_3],\rho(x_1,x_2)a_3+\rho(x_2,x_3)a_1+\rho(x_3,x_1)a_2),(x_4,a_4))\\
&+\bar{\rho}((x_3,a_3),([x_1,x_2,x_4],\rho(x_1,x_2)a_4+\rho(x_2,x_4)a_1+\rho(x_4,x_1)a_2))\\
=&\rho([x_1,x_2,x_3],x_4)+\rho(x_3,[x_1,x_2,x_4])\\
=&[\rho(x_1,x_2),\rho(x_3,x_4)]\\
=&[\bar{\rho}((x_1,a_1),(x_2,a_2)),\bar{\rho}((x_3,a_3),(x_4,a_4))],\\
\end{split}
\end{equation*}
\begin{equation*}
\begin{split}
&\bar{\rho}((x_1,a_1),(x_2,a_2))\bar{\rho}((x_3,a_3),(x_4,a_4))+\bar{\rho}((x_2,a_2),(x_3,a_3))\bar{\rho}((x_1,a_1),(x_4,a_4))\\
&+\bar{\rho}((x_3,a_3),(x_1,a_1))\bar{\rho}((x_2,a_2),(x_4,a_4))\\
=&\rho(x_1,x_2)\rho(x_3,x_4)+\rho(x_2,x_3)\rho(x_1,x_4)+\rho(x_3,x_1)\rho(x_2,x_4)\\
=&\rho([x_1,x_2,x_3],x_4)=\bar{\rho}([(x_1,a_1),(x_2,a_2),(x_3,a_3)],(x_4,a_4)),
\end{split}
\end{equation*}
therefore,  Eq \eqref{eq:mod2} holds,
 and $(A, \bar{\rho})$ is the 3-Lie algebra $E$-module. For all $b\in A$, since
\begin{equation*}
\begin{split}
&[(x_1,a_1),(x_2,a_2),b(x_3,a_3)]=[(x_1,a_1),(x_2,a_2),(bx_3,ba_3)]\\
=&([x_1,x_2,bx_3],\rho(x_1,x_2)(ba_3)+\rho(x_2,bx_3)a_1+\rho(bx_3,x_1)a_2)\\
=&(b[x_1,x_2,x_3]+\rho(x_1,x_2)bx_3, \rho(x_1,x_2)(ba_3)+b\rho(x_2,x_3)a_1+b\rho(x_3,x_1)a_2)\\
=&(b[x_1,x_2,x_3]+\rho(x_1,x_2)bx_3,a_3\rho(x_1,x_2)b+b\rho(x_1,x_2)a_3+b\rho(x_2,x_3)a_1+b\rho(x_3,x_1)a_2)\\
=&b([x_1,x_2,x_3],\rho(x_1,x_2)a_3+\rho(x_2,x_3)a_1+\rho(x_3,x_1)a_2)+(\rho(x_1,x_2)bx_3,a_3\rho(x_1,x_2)b)\\
=&b[(x_1,a_1),(x_2,a_2),(x_3,a_3)]+\bar{\rho}(x_1,x_2)b(x_3,a_3).\\
\end{split}
\end{equation*}
$(E, A, \bar{\rho})$ is a 3-Lie-Rinehart algebra. The proof is complete.
 \end{proof}

Let $L$ be a 3-Lie algebra, $J=\{ x\wedge y ~|~ x, y\in L, [x, y, L]=0\}$. Then $((L\wedge L)/ J, [ , , ]_1)$ is a Lie algebra , where
$(L\wedge L)/ J=\{ \overline{x\wedge y}=x\wedge y+J ~|~ \forall x, y\in L\}$,
  \begin{equation}\label{eq:wedge}
  [\overline{x\wedge y}, \overline{x'\wedge y'}]_1=\overline{[x, y, x']\wedge y'}+\overline{x'\wedge [x, y,  y']},  \quad \forall \overline{x\wedge y}, \overline{x'\wedge y'}\in(L\wedge L)/ J.
  \end{equation}

 Let  $(L, A, \rho)$ be a 3-Lie-Rinehart algebra. Then $(A, \rho)$ is a 3-Lie algebra $L$-module. For  $x, y\in L$, if $[x, y, L]=0$,
 them by \eqref{eq:Rinhart1}, for all $a\in A$ and $z\in L$, we have
 $$\rho(x, y)a z=[x, y, az]-a[x, y, z]=0.$$
 Therefore, $\rho(x, y)=0$, and $\rho_2: ((L\wedge L)/ J, [ , , ]_1)\rightarrow Der(A)$ defined by
 $$
 \rho_2(\overline{x\wedge y})=\rho(x, y), ~~~ \forall x, y\in L
 $$
 is a Lie algebra homomorphism, that is, $(A, \rho_2)$ is a Lie algebra $(L\wedge L)/ J$-module.

    {\bf In the following, $L\wedge L$ denotes  the Lie algebra $((L\wedge L)/ J, [ , , ]_1)$, and $\forall x, y\in L, $ $x\wedge y$ denotes $\overline{x\wedge y}$. And
    \eqref{eq:wedge} can be written as}
    \begin{equation}\label{eq:wedge1}
  [x\wedge y, x'\wedge y']=[x, y, x']\wedge y'+x'\wedge [x, y,  y'],  \quad \forall x,y, x',  y'\in L.
  \end{equation}

\begin{theorem}\label{thm:LL}
  Let $(L,A,\rho)$ be a 3-Lie-Rinehart algebra. Then  $(L\wedge L, \rho_2)$ is a Lie-Rinehart algebra {\small\cite{Cas}}, where for all $x_i\wedge y_i\in  L\wedge  L$, $i=1, 2,$
     \begin{equation}
\label{eq:actionad}
b\cdot x\wedge y=\frac{1}{2}\big((bx)\wedge y+x\wedge(by)\big), ~~~\forall b\in A, x\wedge y\in L\wedge L,
\end{equation}
\begin{equation} \label{eq:trho}
\rho_2: L\wedge L\rightarrow Der(A),  ~\rho_2(x\wedge y)=\rho(x,y), ~ \forall x\wedge y\in  L\wedge L.
   \end{equation}
\end{theorem}

\begin{proof}  Thanks to \eqref{eq:wedge}, $ L\wedge L$ is a Lie algebra. By Definition \ref{defin:rinehart} and \eqref{eq:wedge1}, for all $b, b'\in A$, $x, y\in L$,
since for all $z\in L,$
\begin{equation*}
\begin{split}
&\big((bx)\wedge (b'y)+(b'x)\wedge(by)\big)z\\
=&[bx,b'y,z]+[b'x,by,z]\\
=&b[x,b'y,z]+\rho(b'y,z)bx+b'[x,by,z]+\rho(by,z)b'x\\
=&b'b[x,y,z]+b\rho(z,x)b'y+b'\rho(y,z)bx+b'b[x,y,z]+b'\rho(z,x)by+b\rho(y,z)b'x\\
=&2b'b[x,y,z]+\rho(z,x)(b'b)y+\rho(y,z)(b'b)x\\
=&b'b[z, x,y]+\rho(z,x)(b'b)y+b'b[y, z, x]+\rho(y,z)(b'b)x\\
=&[b'bx,y,z]+[x,b'by,z]=\big((b'bx)\wedge y+x\wedge (b'by)\big))z,
\end{split}
\end{equation*}we have
\begin{equation}\label{eq:bb'xy}
(bx)\wedge (b'y)+(b'x)\wedge(by)=(b'bx)\wedge y+x\wedge (b'by).
\end{equation}

Thanks to Eqs \eqref{eq:actionad} and \eqref{eq:bb'xy}, for all  $b, b'\in A$,
\begin{equation*}
\begin{split}
b'\cdot(b\cdot x\wedge y)=&\frac{1}{2}b'\cdot \big((bx)\wedge y+x\wedge (by)\big)\\
=&\frac{1}{4}\big((b'bx)\wedge y+(bx)\wedge (b'y)+(b'x)\wedge(by)+x\wedge(b'by)\big)\\
=&\frac{1}{2}((b'bx)\wedge y+x\wedge(b'by))=(b'b)\cdot x\wedge y.
\end{split}
\end{equation*}
Then $L\wedge L$ is an $A$-module.

  By Eqs \eqref{eq:wedge1} and  \eqref{eq:trho}, for all $x_1, x_2, y_1, y_2\in L$, $b\in A$, we have
\begin{equation*}
\begin{split}
\rho_2([x_1\wedge y_1, x_2\wedge y_2])=&\rho_2([x_1,y_1,x_2]\wedge y_2+x_2\wedge [x_1,y_1,y_2])\\
=&\rho([x_1,y_1,x_2],y_2)+\rho(x_2,[x_1,y_1,y_2])\\
=&[\rho(x_1,y_1),\rho(x_2,y_2)]=[\rho_2(x_1\wedge y_1),\rho_2(x_2\wedge y_2)],
\end{split}
\end{equation*}
and $
\rho_2(b\cdot x\wedge y)=b\cdot \rho_2(x\wedge y).$ Therefore,
$(A, \rho_2)$ is a Lie algebra $L \wedge L$-module.

For all $x_1, x_2, y_1, y_2, z\in L$, $b\in A$, by Eqs \eqref{eq:Rinhart1}, \eqref{eq:wedge1} and \eqref{eq:actionad},
\begin{equation*}
\begin{split}
&[x_1\wedge y_1, b\cdot x_2\wedge y_2]\\
=&\frac{1}{2}\big([ x_1,  y_1, bx_2]\wedge y_2+bx_2\wedge [x_1,y_1, y_2]+[x_1,y_1, x_2]\wedge by_2+x_2\wedge [x_1,y_1, by_2]\big)\\
=&\frac{1}{2}\big((b[ x_1,  y_1, x_2]+\rho(x_1, y_1)b x_2)\wedge y_2+x_2\wedge (b [x_1,y_1, y_2]+\rho(x_1, y_1)b y_2)\\
&+bx_2\wedge [x_1,y_1, y_2]+[x_1,y_1, x_2]\wedge by_2\big)\\
=&b\cdot ([x_1,y_1,x_2]\wedge y_2+x_2\wedge[x_1,y_1,x_2])+\frac{1}{2}(\rho(x_1, y_1)b x_2\wedge y_2+x_2\wedge\rho(x_1, y_1)b y_2)\\
=&b\cdot [x_1\wedge y_1, x_2\wedge y_2]+\rho_2(x_1\wedge y_1)b \cdot x_2\wedge y_2.\\
\end{split}
\end{equation*}

Therefore,
  $(L\wedge L, \rho_2)$ is a Lie-Rinehart algebra. The proof is complete.
 \end{proof}

    Let $(L, A, \rho)$ be a 3-Lie-Rinehart algebra, $(R, A)$ be a 3-Lie $A$-algebra, and $(L\wedge L, \rho_2)$ be the Lie-Rinehart algebra in Theorem  \ref{thm:LL}.

     Denotes $W(L, R, A)$ be the vector space spanned by the subset of $Der(R)\otimes (L\wedge L)$
    \begin{equation}
    \label{eq:do}
    \begin{split}
    W(L, R, A):=&\{(\varphi, x\wedge y) ~|~\varphi\in Der(R), x,y\in L, ~\\
    &\mbox{satisfying}~ \forall b\in A, r\in R, \varphi(br)=b\varphi(r)+ \rho(x, y)br\}.\\
    \end{split}
     \end{equation}

  Then we have the following result.

 \begin{theorem}\label{thm:do}
  Let $(L,A,\rho)$ be a 3-Lie-Rinehart algebra and $(R, A)$ be a 3-Lie $A$-algebra.
   Then $(W(L, R, A), \rho_3)$
   is a Lie-Rinehart algebra, where for , $\forall (d, x\wedge y), (d_i, x_i\wedge y_i)\in  W(L, R, A),i=1, 2,$

   \begin{equation}\label{eq:do1}
  b(\varphi, x\wedge y)=(b\varphi, b\cdot x\wedge y), ~~ \forall  b\in A,
    \end{equation}
    \begin{equation}\label{eq:do2}
    \rho_3:W(L, R, A)\rightarrow Der(A),~~
    \rho_3(\varphi, x\wedge y)=\rho(x, y),
    \end{equation}
    \begin{equation}
    \label{eq:domul}
    [(\varphi_1,x_1\wedge y_1),~(\varphi_2,x_2\wedge y_2)]=([\varphi_1,\varphi_2],[x_1\wedge y_1,x_2\wedge y_2]).
   \end{equation}
   \end{theorem}

\begin{proof}   By Eqs \eqref{eq:actionad},  \eqref{eq:do} and \eqref{eq:do1}, $W(L, R, A)$ is an $A$-module .

Thanks to  Theorem  \ref{thm:LL} and a direct computation,  $W(L, R, A)$ is a Lie algebra with the multiplication \eqref{eq:domul}.
By Eq \eqref{eq:do2}, for all $(\varphi_1,x_1\wedge y_1),(\varphi_2,x_2\wedge y_2)\in W(L,R,A)$,
 \begin{equation*}
\begin{split}
&\rho_3[(\varphi_1,x_1\wedge y_1),(\varphi_2,x_2\wedge y_2)]
=\rho_3([\varphi_1,\varphi_2],[x_1\wedge y_1,x_2\wedge y_2])=\rho_2([x_1\wedge y_1,x_2\wedge y_2])\\
=&[\rho_2(x_1\wedge y_1),\rho_2(x_2\wedge y_2)]=[\rho_3(\varphi_1,x_1\wedge y_1),\rho_3(\varphi_2,x_2\wedge y_2)],\\
\end{split}
\end{equation*}
it follows that $A$ is a Lie algebra $W(L, R, A)$-module and $\rho_3(W(L, R, A))\subseteq Der(A)$.

  For any $b\in A$, $(\varphi_1,x_1\wedge y_1),(\varphi_2,x_2\wedge y_2)\in W(L,R,A)$, by Theorem \ref{thm:LL} and \eqref{eq:domul},
\begin{equation*}
\begin{split}
&[(\varphi_1,x_1\wedge y_1),b(\varphi_2,x_2\wedge y_2)]\\
=&[(\varphi_1,x_1\wedge y_1),(b\varphi_2, b\cdot x_2\wedge y_2)]\\
=&([\varphi_1,b\varphi_2],[x_1\wedge y_1,b\cdot x_2\wedge y_2])\\
=&(\varphi_1(b\varphi_2)-(b\varphi_2)\varphi_1,~b\cdot [x_1\wedge y_1, x_2\wedge y_2]+\rho_2(x_1\wedge y_1)b \cdot (x_2\wedge y_2))\\
=&(b(\varphi_1\varphi_2)-(b\varphi_2)\varphi_1+\rho_2(x_1\wedge y_1)b \varphi_2,~b\cdot [x_1\wedge y_1,x_2\wedge y_2]+\rho_2(x_1\wedge y_1)b\cdot x_2\wedge y_2)\\
=&b([\varphi_1,\varphi_2],[x_1\wedge y_1,x_2\wedge y_2])+\rho_2(x_1\wedge y_1)b (\varphi_2,x_2\wedge y_2)\\
=&b[(\varphi_1,x_1\wedge y_1),(\varphi_2,x_2\wedge y_2)]+\rho_3(\varphi_1,x_1\wedge y_1)b (\varphi_2,x_2\wedge y_2).\\
\end{split}
\end{equation*}

  Therefore, $(W(L, R, A), \rho_3)$ is a Lie-Rinehart algebra.
 \end{proof}

\begin{coro}
  Let $(L,A,\rho)$  be a 3-Lie-Rinehart algebra, and $(R, A)$ be a 3-Lie A-algebra,
  then
  $ W(L, R\wedge R, A)  $
is a subalgebra of the Lie-Rinehart algebra
 $W(L, R, A)  $, where

 \begin{equation*}
\begin{split} W(L, R\wedge R, A)=&\{(r_1\wedge r_2, x_1\wedge x_2) ~|~  x_1, x_2\in L, r_1, r_2\in R,\\
&\mbox{ satisfying}~~~
(r_1\wedge r_2)(br)=b [r_1, r_2, r]+\rho(x_1, x_2)b r, \forall r\in R, b\in A\}.
 \end{split}
\end{equation*}
\end{coro}

\begin{proof} Apply Theorem \ref{thm:do} and Eq \eqref{eq:do}.
 \end{proof}

\section{actions and modules of 3-Lie-Rinehart algebras}

  In this section, we discuss actions and modules of 3-Lie-Rinehart algeras.

\begin{defn}\label{defin:action}
  Let $(L, A, \rho)$ be a 3-Lie-Rinehart algebra, and $(R, A)$ be a 3-Lie $A$-algebra and $\beta:L\wedge L\rightarrow Der(R) $ be a $F$-linear mapping. We say that {\bf  $(L, A, \rho)$ acts on  $(R, A)$ by $\beta$} (or $\beta$ {\bf  is an action of $(L, A, \rho)$ on $(R, A)$} ), if following properties hold:

$(1)$ $(R,\beta)$ is a 3-Lie algebra L-module,

$(2)$ $ \beta(ax, y )=\beta(x, ay)=a \beta(x,y)$, $\forall a\in A$, $x, y\in L,$

$(3)$ $ \beta(x,y)(ar)=a\beta(x,y)r+\rho(x,y)(a)r$, $\forall a\in A, r\in R, x, y\in L.$

Furthermore, if $R$ is abelian, that is, $[R, R, R]=0$,  then $(R,\beta)$ is called {\bf a 3-Lie-Rinehart algebra L-module} (or, $(R,\beta)$ is an $(L, A, \rho)$-module).
\end{defn}

\begin{defn}\label{defin:modhom}
  Let $(L, A, \rho)$ be a 3-Lie-Rinehart algebra, $(R_1, A)$ and $(R_2, A)$ be abelian 3-Lie $A$-algebras, and
   $(R_1, \beta_1), (R_2, \beta_2)$ be  $(L, A, \rho)$-modules. For an $A$-module homomorphism $f: R_1\rightarrow R_2$, if $f$ satisfies
\begin{equation*}
  f(\beta_1(x,y)r)=\beta_2(x,y)f(r), ~~  \forall x, y\in L, r\in R_1,
\end{equation*}
then $f$ is called {\bf  a 3-Lie-Rinehart module homomorphism} from $(R_1, \beta_1)$ to $(R_2, \beta_2)$, and if $f$ is an isomorphism, then we say that $(R_1, \beta_1)$ is equivalent to $ (R_2, \beta_2)$.
\end{defn}

\begin{theorem}\label{thm:induced}
  Let $(L, A, \rho)$ and  $(L', A, \rho')$ be 3-Lie-Rinehart algebras, $(R, A)$ be an abelian  3-Lie $A$-algebra and $(R,\beta)$ be an $(L, A, \rho)$-module. If $f: L'\rightarrow L$ is a 3-Lie-Rinehart  homomorphism,  then $(R, \beta')$ is an $(L', A, \rho')$-module, where $\beta': L'\wedge L'\rightarrow gl(R)$ is defined by
 \begin{equation}\label{eq:inducedbeta}
  \beta'(x',y')=\beta(f(x'),f(y')), ~~ \forall x',y'\in L',
 \end{equation}
  which is called  {\bf  an $(L', A, \rho')$-module induced by $f$}.
\end{theorem}

\begin{proof}
Since  $(R,\beta)$ is  an $(L, A, \rho)$-module, by Eqs \eqref{eq:mod1} and  \eqref{eq:mod2}, $\forall x'_i\in L', 1\leq i\leq 4,$
\begin{equation*}
\begin{split}
&[\beta'(x_1',x_2'),\beta'(x_3',x_4')]\\
=&[\beta(f(x_1'),f(x_2')),\beta(f(x_3'),f(x_4'))]\\
=&\beta([f(x_1'),f(x_2'),f(x_3')],f(x_4))+\beta(f(x_3'),[f(x_1'),f(x_2'),f(x_4')])\\
=&\beta(f([x_1',x_2',x_3']),f(x_4))+\beta(f(x_3'),f([x_1',x_2',x_4']))\\
=&\beta'([x_1',x_2',x_3'],x_4')+\beta'(x_3',[x_1',x_2',x_4']).\\
\end{split}
\end{equation*}

  \begin{equation*}
\begin{split}
\beta'([x_1',x_2',x_3'],x_4')=&\beta'(x_1',x_2')\beta'(x_3',x_4')+\beta'(x_2',x_3')\beta'(x_1',x_4')\\
  &+\beta'(x_3',x_1')\beta'(x_2',x_4').
\end{split}
\end{equation*}

Therefore, $(R,\beta')$ is  a 3-Lie algebra $L'$-module.

 Since $f: L'\rightarrow L$ is a 3-Lie-Rinehart homomorphism, $\forall x',y'\in L', a\in A$,
  \begin{equation*}
  \beta'(ax',y')=\beta(f(ax'),f(y'))=\beta(af(x'),f(y'))=a\beta'(x',y'),
  \end{equation*}
  \begin{equation*}
  \beta'(x',ay')=\beta'(ax',y')=a\beta'(x',y'),
  \end{equation*}
\begin{equation*}
\beta'(x',y')(ar)=a\beta'(x',y')r+\rho'(x',y')(a)r.
\end{equation*}

  Therefore, $(R, \beta')$ is an $(L', A, \rho')$-module.
\end{proof}

\begin{theorem}\label{thm:module}
  Let $(L, A, \rho)$ be a 3-Lie-Rinehart algebra, $(R, A)$ be an abelian  3-Lie $A$-algebra, and $(R,\beta)$ be a 3-Lie algebra $L$-module.
  Then $(R,\beta)$ is an  $(L, A, \rho)$-module if and only if $(L\ltimes R, A, \rho_4)$ is a 3-Lie-Rinehart algebra with the multiplication
  $ ~~ \forall x_1, x_2, x_3\in L, r_1, r_2, r_3\in R,$
   \begin{equation}
\label{eq:semiproductR} [x_1+r_1, x_2+r_2, x_3+r_3]=[x_1, x_2, x_3]+\beta(x_1, x_2)r_3+\beta(x_2, x_3)r_1+\beta(x_3, x_1)r_2,
\end{equation} and
  \begin{equation}\label{eq:oplus1}
  \rho_4: (L\ltimes R)\wedge(L\ltimes R) \rightarrow Der(A),~~\rho_4(x_1+r_1,x_2+r_2)=\rho(x_1,x_2).
  \end{equation}
\end{theorem}

\begin{proof} Since  $L$ and $R$ are $A$-modules,
  $L\ltimes R=L\dot+ R$ is  an $A$-module and satisfies
   $$a(x+r)=ax+ar, ~~~ \forall a\in A, x\in L, r\in R.$$

  If $(R,\beta)$ is an  $(L, A, \rho)$-module. Then by  Eqs \eqref{eq:semiproductR} and \eqref{eq:oplus1}, and Definition \ref{defin:action},
 $L\ltimes R$ is a 3-Lie algebra with the multiplication \eqref{eq:semiproductR},  and $(A, \rho_4)$ is a 3-Lie algebra $L\ltimes R$-module. And for
 $\forall x_i\in L, r_i\in R$ and $a\in A,$ $i=1,2,3, $
  \begin{equation*}
\begin{split}
&[x_1+r_1, x_2+r_2, a(x_3+r_3)]=[x_1+r_1, x_2+r_2, ax_3]+[x_1+r_1, x_2+r_2, ar_3]\\
=&[x_1, x_2, ax_3]+\beta(x_1,x_2)(ar_3)+\beta(ax_3, x_1)r_2+\beta(x_2,ax_3)r_1\\
=&a([x_1,x_2,x_3]+\beta(x_1,x_2)r_3+\beta(x_2,x_3)r_1+\beta(x_3,x_1)r_2)+\rho(x_1,x_2)ax_3+\rho(x_1,x_2)a r_3\\
=&a[x_1+r_1,x_2+r_2,x_3+r_3]+\rho_4(x_1+r_1,x_2+r_2)a(x_3+r_3).\\
\end{split}
 \end{equation*}
  Therefore, $(L\ltimes R, A, \rho_4)$ is a 3-Lie-Rinehart algebra.

  Conversely, if $(L\ltimes R, A, \rho_4)$ is a 3-Lie-Rinehart algebra. Then by \eqref{eq:semiproductR}, for all $a\in A, x_1, x_2\in L, r,r_1, r_2, r_3\in R$,
  \begin{equation*}
\begin{split}
\beta(x_1, x_2)[r_1, r_2, r_3]=&[x_1, x_2, [r_1, r_2, r_3]]\\
=&[[x_1, x_2, r_1], r_2, r_3]+[r_1, [x_1, x_2, r_2], r_3]+[r_1, r_2, [x_1, x_2, r_3]]\\
 =&[\beta(x_1, x_2)r_1, r_2, r_3]+[r_1, \beta(x_1, x_2)r_2, r_3]+[r_1, r_2, \beta(x_1, x_2)r_3],
 \end{split}
\end{equation*}
$$\beta(x_1, x_2)(ar)=[x_1, x_2, ar]=a[x_1, x_2, r]+\rho_4(x_1, x_2)a~~r=a\beta(x_1, x_2)r+\rho(x_1, x_2)a~r,$$
  we get $\beta(L\wedge L)\subseteq Der(R)$, and $\beta$
satisfies Definition \ref{defin:action}. Therefore,  $(R,\beta)$ is a 3-Lie-Rinehart $(L, A, \rho)$-module.
\end{proof}

\begin{coro}\label{coro:pai}
  Let $(L, A, \rho)$ be a 3-Lie-Rinehart algebra, and $(R,\beta)$ be an $(L, A, \rho)$-module. Then $F$-linear mappings
 $ ~~ \pi: L\ltimes R\rightarrow L$ and $  i: R\rightarrow L\ltimes R ~~ ~$ defined by
 \begin{equation}\label{eq:pii}
  ~~\pi(x+r)=x, ~~~~~
 i(r)=r,~~ \forall  x\in L,~r\in R,
   \end{equation}
are $3$-Lie-Rinehart homomorphisms.
\end{coro}

\begin{proof} Thanks to Theorem \ref{thm:module},
  $\pi$ and $i$ are $A$-module homomorphisms.

  By Eq \eqref{eq:semiproduct}, for all $x_i\in L$, $r_i\in R$,
 \begin{equation*}
\begin{split}
\pi[x_1+r_1,x_2+r_2,x_3+r_3]=&\pi ([x_1,x_2,x_3]+\beta(x_1,x_2)r_3+\beta(x_2,x_3)r_1+\beta(x_3,x_1)r_2)\\
=&[x_1,x_2,x_3]=[\pi(x_1+r_1),\pi(x_2+r_2),\pi(x_3+r_3)].\\
\end{split}
\end{equation*}
  Then $\pi$ is a 3-Lie algebra homomorphism.

Thanks to  Eq \eqref{eq:oplus1},
 \begin{equation*}
 \rho_4(x_1+r_1 ,x_2+r_2)=\rho(x_1,x_2)=\rho(\pi(x_1+r_1),\pi(x_2+r_2)).
\end{equation*}
   Therefore, $\pi$ is a 3-Lie-Rinehart homomorphisms.

   By the similar discussion to the above, $i$ is a 3-Lie-Rinehart homomorphism.
\end{proof}

{\bf Remark }  Let $(L, A, \rho)$ be a 3-Lie-Rinehart algebra. Then  the regular representation $(R, \mbox{ad})$ of $L$ may not a  $(L, A, \rho)$-module, where $R$ is the copy of $L$ as an abelian 3-Lie A-algebra.
But if $ \rho=0$, then it is easy to check that $(R, \mbox{ad})$ is a  $(L, A, 0)$-module.

In fact, for  $\forall x, y\in L, z\in R=L$ and $b\in A$,

$\mbox{ad}(bx,y)z=[b x,y, z]=b[x,y,z]+\rho(y,z)b x=b ~~\mbox{ad}(x,y)+\rho(y,z)b x\neq b ~~\mbox{ad}(x,y)z$,
\\therefore,  $(R, \mbox{ad})$ satisfies Definition \ref{defin:action} if and only if $\rho=0$.

\begin{theorem}\label{thm:action1}
  Let $(L, A, \rho)$ be a 3-Lie-Rinehart algebra, $(R, A)$ be a 3-Lie $A$-algebra, and $f: L\rightarrow R$ be a 3-Lie $A$-algebra homomorphism, that is, $f$ is a 3-Lie algebra homomorphism and satisfies
  $f(ax)=af(x)$  for all $a\in A, x\in L$. Then $f$ induces an action $\beta$ of $(L, A, \rho)$ on $(R, A)$, where   $\beta: L\wedge L\rightarrow gl(R),$
  \begin{equation}\label{eq:action1}
    \beta(x,y)r=[f(x),f(y),r],~~ \forall x,y\in L, r\in R.
  \end{equation}
\end{theorem}

\begin{proof} By Eq \eqref{eq:action1},
 $\forall x,y\in L, r, r_1, r_2, r_3\in R, a\in A,$
 \begin{equation*}
\begin{split}
\beta(x,y)[r_1,r_2,r_3]=&[f(x),f(y),[r_1,r_2,r_3]]\\
=&[[f(x),f(y),r_1],r_2,r_3]+[r_1,[f(x),f(y),r_2],r_3]+[r_1,r_2,[f(x),f(y),r_3]]\\
=&[\beta(x,y)r_1,r_2,r_3]+[r_1,\beta(x,y)r_2,r_3]+[r_1,r_2,\beta(x,y)r_3].\\
 \end{split}
\end{equation*}
 Then $\beta(L\wedge L)\subseteq Der(R)$. Since $f$ is a  3-Lie $A$-algebra homomorphism,
    \begin{equation*}
\begin{split}
[\beta(x_1, y_1), \beta(x_2, y_2)]=&[\mbox{ad}(f(x_1),f(y_1)),\mbox{ad}(f(x_2),f(y_2))]\\
=&\mbox{ad}([f(x_1),f(y_1), f(x_2)], f(y_2))+\mbox{ad}(f(x_2),[f(x_1),f(y_1),f(y_2)])\\
=&\mbox{ad}(f([x_1, y_1, x_2]), f(y_2))+\mbox{ad}[f(x_2),f([x_1,y_1,y_2]))\\
=&\beta([x_1,y_1, x_2], y_2)+\beta(x_2,[x_1,y_1,y_2]).\\
 \end{split}
\end{equation*}
 \begin{equation*}
\begin{split}
\beta([x_1, x_2 x_3],y) =&\mbox{ad}(f([x_1, x_2, x_3]), f(y))
=\mbox{ad}([f(x_1),f(x_2), f(x_3)], f(y))\\
=&\mbox{ad}(f(x_1), f(x_2))\mbox{ad}( f(x_3), f(y))+\mbox{ad}(f(x_2), f(x_3))\mbox{ad}( f(x_1), f(y))\\
&+\mbox{ad}(f(x_3), f(x_1))\mbox{ad}( f(x_2), f(y))\\
=&\beta(x_1, x_2)\beta(x_3, y)+\beta(x_2, x_3)\beta( x_1, y)+\beta(x_3, x_1)\beta( x_2, y),\\
 \end{split}
\end{equation*}
it follows that $(R, \beta)$ is a 3-Lie algebra $L$-module.
Thanks to \begin{equation*}
\beta(ax,y)r=[f(ax),f(y),r]=[af(x),f(y),r]=a[f(x),f(y),r]=a\beta(x,y)r=\beta(x,ay)r,
\end{equation*}
 $\beta$ is an action of $(L, A, \rho)$ on $(R, A)$.
\end{proof}

 By Theorem \ref{thm:LL} and Theorem \ref{thm:do}, for any 3-Lie-Rinehart algebra $(L, A, \rho)$ and a 3-Lie $A$-algebra $R$, $(L\wedge L, \rho_2)$ and  $(W(L,R,A), \rho_3)$ are Lie-Rinehart algebras, where  $\rho_2: L\wedge L\rightarrow Der(A)$ and  $\rho_3:
 W(L,R,A)\rightarrow Der(A)$ are defined by  \eqref{eq:trho} and~ \eqref{eq:do2}, respectively.

   Now we consider an equivalent  relation on $W(L,R,A)$ as follows, for all $(\varphi,x\wedge y), (\varphi',x'\wedge y')\in (W(L,R,A),  \rho_3),$
\begin{equation}\label{eq:relation}
(\varphi,x\wedge y)\sim (\varphi',x'\wedge y')\Leftrightarrow \varphi=\varphi' ~\mbox{and}~ \rho_2(x\wedge y)=\rho_2(x'\wedge y').
 \end{equation}
Denotes
\begin{equation*}
\overline{(\varphi,x\wedge y)}:=\big\{(\varphi',x'\wedge y')~|~ (\varphi',x'\wedge y')\in W(L,R,A) ~\mbox{and}~(\varphi',x'\wedge y')\sim (\varphi,x\wedge y)\big\},
\end{equation*}
$$\overline{W(L,R,A)}=W(L,R,A)/\sim=\{\overline{(\varphi,x\wedge y)}~|~\forall (\varphi,x\wedge y)\in W(L,R,A)\}.
$$
Define linear multiplication $[ ~,~ ]: \overline{W(L,R,A)}\wedge \overline{W(L,R,A)}\rightarrow \overline{W(L,R,A)} $ by
\begin{equation}\label{eq:wmultiplication}
 [\overline{(\varphi_1,x_1\wedge y_1)},\overline{(\varphi_2,x_2\wedge y_2)}]=\overline{([\varphi_1,\varphi_2],[x_1\wedge y_1,x_2\wedge y_2])}.
\end{equation}

 By the above notations, we have the following result.

\begin{lemma}\label{lem:barWar}
$(\overline{W(L,R,A)}, \tilde{\rho_3})$ is a Lie-Rinehart algebra with the multiplication \eqref{eq:wmultiplication}, where
\begin{equation}\label{eq:rhow}
\tilde{\rho_3}: \overline{W(L,R,A)}\rightarrow Der(A),~~ \tilde{\rho_3}(\overline{(\varphi, x\wedge y)})=\rho(x, y),
\end{equation}
for all $x, y\in L$.
\end{lemma}

\begin{proof} By \eqref{eq:do} and \eqref{eq:relation}, $\tilde{\rho_3}$ is valid. Thanks to Theorem \ref{thm:do}, we get the result.  \end{proof}

\begin{theorem}\label{thm:actioneta}
  Let $(L,A,\rho)$ be a 3-Lie-Rinehart algebra and $(R,A)$ be a Lie A-algebra. Then there is an action $\beta$ of Lie-Rinehart algebra $(L\wedge L,  \rho_2)$ on $(R, A)$ if and only if there exists a Lie-Rinehart algebra homomorphism $\eta: (L\wedge L, \rho_2)\rightarrow (\overline{W(L,R,A)}, \tilde{\rho_3})$.
\end{theorem}

\begin{proof} By Lemma \ref{lem:barWar}, $(\overline{W(L,R,A)}, \tilde{\rho_3})$ is a Lie-Rinehart algebra.

If $\eta: (L\wedge L, \rho_2)\rightarrow (\overline{W(L,R,A)}, \tilde{\rho_3})$ is  a Lie-Rinehart algebra homomorphism. For all $x, y\in L,$ denotes
\begin{equation}\label{eq:eta}
\eta(x\wedge y)=\overline{(\varphi_{x\wedge y},x\wedge y)} \in \overline{W(L, R, A)}.
\end{equation}

Thanks to Eq \eqref{eq:relation}, we can define linear mapping $ \beta: L\wedge L\rightarrow Der(R)$
by, $~\forall x\wedge y\in L\wedge L,$
   \begin{equation}\label{eq:beta}
   \beta(x, y):=\varphi_{x\wedge y},~~  ~~\mbox{where}~~\overline{(\varphi_{x\wedge y},x\wedge y)} \in \overline{W(L, R, A)}.
  \end{equation}
Thanks to \eqref{eq:relation}, $\beta$ is valid. Since $\eta$ is an algebra homomorphism, and \eqref{eq:wedge1} and \eqref{eq:eta},  we have $\forall x_1, x_2, y_1, y_2\in L$,

 \begin{equation*}
 \begin{split}
 &\eta[x_1\wedge y_1,x_2\wedge y_2]\\
 =&\eta([x_1,y_1,x_2]\wedge y_2+x_2\wedge [x_1,y_1,y_2])\\
 =&\overline{(\varphi_{[x_1,y_1,x_2]\wedge y_2}+\varphi_{x_2\wedge [x_1,y_1,y_2]},[x_1,y_1,x_2]\wedge y_2+x_2\wedge [x_1,y_1,y_2])}\\
 =&\overline{(\varphi_{[x_1,y_1,x_2]\wedge y_2}+\varphi_{x_2\wedge [x_1,y_1,y_2]},[x_1\wedge y_1,x_2\wedge y_2])}\\
 =&[\eta(x_1\wedge y_1),\eta(x_2\wedge y_2)]\\
 =&[\overline{(\varphi_{x_1\wedge y_1},x_1\wedge y_1)},\overline{(\varphi_{x_2\wedge y_2},x_2\wedge y_2)}]\\
 =&\overline{([\varphi_{x_1\wedge y_1},\varphi_{x_2\wedge y_2}],[x_1\wedge y_1,x_2\wedge y_2])}.
 \end{split}
 \end{equation*}
Then we have
\begin{equation*}
[\varphi_{x_1\wedge y_1},\varphi_{x_2\wedge y_2}]=\varphi_{[x_1,y_1,x_2]\wedge y_2}+\varphi_{x_2\wedge [x_1,y_1,y_2]}.
\end{equation*}
 Thanks to Eq \eqref{eq:beta},
\begin{equation*}
\begin{split}
&[\beta(x_1\wedge y_1),\beta(x_2\wedge y_2)]=[\varphi_{x_1\wedge y_1},\varphi_{x_2\wedge y_2}]=\varphi_{[x_1,y_1,x_2]\wedge y_2}+\varphi_{x_2\wedge [x_1,y_1,y_2]}\\
=&\beta([x_1,y_1,x_2]\wedge y_2)+\beta(x_2\wedge [x_1,y_1,y_2])=\beta[x_1\wedge y_1,x_2\wedge y_2],\\
\end{split}
\end{equation*}
and $(R,\beta)$ is a Lie algebra $L\wedge L$-module.

Since $\eta$ is an $A$-module homomorphism and  Eqs \eqref{eq:actionad} and \eqref{eq:eta}, we hve
$$\eta(b\cdot x\wedge y)=\frac{1}{2}\overline{(\varphi_{bx\wedge y}+\varphi_{x\wedge by},bx\wedge y+x\wedge by)}=\overline{(b\cdot \varphi_{x\wedge y},b\cdot x\wedge y)}=b\eta(x\wedge y),
$$
for all $x,y\in L, b\in B$. Therefore,
\begin{equation*}
b\cdot \varphi_{x\wedge y}=\frac{1}{2}\varphi_{bx\wedge y}+\frac{1}{2}\varphi_{x\wedge by},
\end{equation*}
and
 \begin{equation*}
 \begin{split}
 &\beta(b\cdot x\wedge y) =\beta(\frac{1}{2}bx\wedge y+\frac{1}{2}x\wedge by)=\frac{1}{2}\beta(bx\wedge y)+\frac{1}{2}\beta(x\wedge by)\\
 =&\frac{1}{2}\varphi_{bx\wedge y}+\frac{1}{2}\varphi_{x\wedge by} =b\cdot \varphi_{x\wedge y} =b\beta(x\wedge y).
 \end{split}
  \end{equation*}

  By Eqs \eqref{eq:do} and \eqref{eq:beta}, for all $x, y\in L$, $b\in A$ and $r\in R,$
\begin{equation*}
 \begin{split}
\beta(x\wedge y)(br)=&\varphi_{x\wedge y}(br)=b\varphi_{x\wedge y}r+\rho_2(x\wedge y)br=b\beta(x\wedge y)r+\rho_2(x\wedge y)br.
\end{split}
  \end{equation*}

  Therefore, $\beta$ is an action of Lie-Rinehart algebra $(L\wedge L, \rho_2)$ on $(R, A)$.

  Conversely, suppose $\beta: L\wedge L\rightarrow Der(R)$  is an action of Lie-Rinehart algebra $(L\wedge L, \rho_2)$ on $(R, A)$. Define linear mapping
 \begin{equation*}
 \eta: L\wedge L \rightarrow \overline{W(L,R,A)},~~
 \eta(x\wedge y)=\overline{(\beta(x\wedge y),x\wedge y)}, ~~\forall x, y\in L.
  \end{equation*}
Since $\beta(x\wedge y) (br)=b\beta(x\wedge y)r+\rho_2(x\wedge y)b r=b\beta(x\wedge y)r+\rho(x, y)b r$, and Eq \eqref{eq:relation}, $\eta$ is valid.    By  \eqref{eq:wedge1} and \eqref{eq:actionad},
 \begin{equation*}
 \begin{split}
 &\eta[x_1\wedge y_1,x_2\wedge y_2]\\
 =&\eta([x_1,y_1,x_2]\wedge y_2+x_2\wedge [x_1,y_1,y_2])\\
 =&\overline{(\beta([x_1,y_1,x_2]\wedge y_2)+\beta(x_2\wedge [x_1,y_1,y_2]),[x_1,y_1,x_2]\wedge y_2+x_2\wedge [x_1,y_1,y_2])}\\
 =&\overline{(\beta([x_1,y_1,x_2]\wedge y_2)+\beta(x_2\wedge [x_1,y_1,y_2]),[x_1\wedge y_1,x_2\wedge y_2])}\\
 =&\overline{([\beta(x_1\wedge y_1),\beta(x_2\wedge y_2)],[x_1\wedge y_1,x_2\wedge y_2])}\\
  =&[\overline{(\beta(x_1\wedge y_1),x_1\wedge y_1)},\overline{(\beta(x_2\wedge y_2),x_2\wedge y_2)}]\\
 =&[\eta(x_1\wedge y_1),\eta(x_2\wedge y_2)].
 \end{split}
  \end{equation*}

  It follows that $\eta$ is a Lie algebra homomorphism. Thanks to  \eqref{eq:actionad},  $\eta$ is an $A$-module homomorphism and satisfies
\begin{equation*}
\tilde{\rho_3}(\eta(x\wedge y))=\rho_2(x\wedge y)=\rho(x, y), \forall x, y\in L.
\end{equation*}

Therefore $\eta$ is a Lie-Rinehart algebra homomorphism. The proof is complete.
 \end{proof}

\section{Derivation of 3-Lie-Rinehart algebras}

In this section, we discuss the  derivations from 3-Lie-Rinehart algebras over $A$ to $3$-Lie $A$-algebras.

\begin{defn}\label{defin:derivation}
  Let $(L, A, \rho)$ be a 3-Lie-Rinehart algebra, $(R, A)$ be a 3-Lie $A$-algebra and $\beta: L\wedge L\rightarrow Der(R)$ be an action of $(L,  A, \rho)$ on $(R, A)$.  If an $A$-linear mapping  $\psi: L\rightarrow R$ has the property:  for all $x, y, z\in L$,
\begin{equation}\label{eq:derivation}
\psi([x,y,z])=[\psi(x), \psi(y), \psi(z)]+\beta(x,y)\psi(z)+\beta(y,z)\psi(x)+\beta(z,x)\psi(y),
\end{equation}
then $\psi$ is called {\bf a 3-Lie-Rinehart derivation from $(L, A, \rho)$ to $(R, A)$}, and  $Der_{\beta}(L, R)$ denotes all 3-Lie-Rinehart derivations from $(L, A, \rho)$ to $(R, A)$.
\end{defn}

\begin{theorem}\label{thm:der}
Let $(L, A, \rho)$ be a 3-Lie-Rinehart algebra, $(R, A)$ be a 3-Lie $A$-algebra, $\beta: L\wedge L\rightarrow gl(R)$ be an action of $(L,  A, \rho)$ on $(R, A)$, and $\psi\in Der_{\beta}(L, R)$.
  Then $(R, \bar{\beta})$ is a 3-Lie algebra $ L\ltimes R$-module, and   $\bar{\psi}\in Der_{\bar{\beta}}(L\ltimes R, R)$, where
  $\bar{\beta}: (L\ltimes R)\wedge (L\ltimes R) \rightarrow gl(R),$
  $\bar{\psi}: L\ltimes R\rightarrow R$ defined by
  \begin{equation}\label{eq:barbeta}
 \bar{\beta}(x_1+r_1, x_2+r_2)=\beta(x_1, x_2), ~~~ \bar{\psi}(x+r)=\psi(x), ~~ \forall x, x_1, x_2\in L, r, r_1, r_2\in R,
  \end{equation}
  and  $\bar{\psi}$ is also a derivation of the 3-Lie algebra $L\ltimes R$ ($\bar{\psi}: L\ltimes R\rightarrow R\subset L\ltimes R $).
\end{theorem}

\begin{proof} Thanks to Theorem \ref{thm:module}, $(L\ltimes R, A, \rho_4)$ is a 3-Lie-Rinehart algebra.  Apply Eq \eqref{eq:semiproduct} and Definition \ref{defin:derivation},
$(R, \bar{\beta})$ is  a 3-Lie algebra $L\ltimes R$-module.
Thanks to Eqs \eqref{eq:semiproduct} and \eqref{eq:derivation}, $\forall   x_1, x_2, x_3\in L, ~r_1,r_2,r_3\in R$,
 \begin{equation*}
\begin{split}
&\bar{\psi}([x_1+r_1, x_2+r_2, x_3+r_3])=\bar{\psi}([x_1,x_2,x_3]+\beta(x_1,x_2)r_3+\beta(x_2,x_3)r_1+\beta(x_3,x_1)r_2\big)\\
=&[\psi(x_1),\psi(x_2), \psi(x_3)]+\beta(x_1,x_2)\psi(x_3)+\beta(x_2,x_3)\psi(x_1)+\beta(x_3,x_1)\psi(x_2)\\
=&[\bar{\psi}(x_1+r_1),\bar{\psi}(x_2+r_2),\bar{\psi}(x_3+r_3)]+\bar{\beta}(x_1+r_1,x_2+r_2)\bar{\psi}(x_3+r_3)\\
&+\bar{\beta}(x_2+r_2, x_3+r_3)\bar{\psi}(x_1+r_1)+\bar{\beta}(x_3+r_3, x_1+r_1)\bar{\psi}(x_2+r_2).\\
\end{split}
\end{equation*}
  Therefore,  $\bar{\psi}\in Der_{\bar{\beta}}(L\ltimes R, R)$  and
  \begin{equation*}
\begin{split}
\bar{\psi}([x_1+r_1, x_2+r_2, x_3+r_3])=&[\bar{\psi}(x_1+r_1),x_2+r_2, x_3+r_3]]+[x_1+r_1,\bar{\psi}(x_2+r_2), x_3+r_3]\\
&+[x_1+r_1, x_2+r_2, \bar{\psi}(x_3+r_3)].\\
\end{split}
\end{equation*}

It follows that $\bar{\psi}$ is also a derivation of the 3-Lie algebra $L\ltimes R$.
 \end{proof}

\begin{coro}\label{coro:der1}
  Let $(L,A,\rho)$ be a 3-Lie-Rinehart algebra,  $R=L$ (the copy of $L$ as an $A$-module) be  an abelian 3-Lie $A$-algebra, and  $(R, \mbox{ad})$ be the regular representation of the 3-Lie algebra $L$. Then for $F$-linear mapping
   $\psi: L\rightarrow R$, $\psi\in Der_{\mbox{ad}}(L, R)$  if and only if $\psi\in Der_A(L)$, where $Der_A(L)$ is defined by \eqref{eq:Ader}.
 \end{coro}

\begin{proof}
Apply Theorem \ref{thm:der}.
 \end{proof}

Let $(L, A, \rho)$ be a 3-Lie-Rinehart algebra, $(R, A)$ be a 3-Lie $A$-algebra and $\beta: L\wedge L\rightarrow gl(R)$ be an action of $(L,  A, \rho)$ on $(R, A)$.  By Theorem \ref{thm:der}, $(R, \bar{\beta})$ is a 3-Lie algebra $L\ltimes R$-module, where
$\bar{\beta}$ is defined by \eqref{eq:barbeta}. Furthermore, If $R$ is an abelian 3-Lie algebra, we get the following result.

\begin{theorem}\label{thm:der2}
 Let $(L, A, \rho)$ be a 3-Lie-Rinehart algebra, $(R, \beta)$ be an $(L, A, \rho)$-module. Then  for all $ x\in L, r\in R$, $\psi(x, r)\in Der_{\bar{\beta}}(L\ltimes R,R)$, where
 \begin{equation}\label{eq:semider}
  \psi(x, r): L\ltimes R\rightarrow R, \quad \psi(x,r)(z+r')=\beta(z,x)r, ~~  \forall  z\in L, r\in R.
 \end{equation}
\end{theorem}

 \begin{proof}  Since $(R, \beta)$ be an $(L, A, \rho)$-module, $R$ is an abelian 3-$A$-Lie algebra, $[R, R, R]=0$. By Theorem \ref{thm:der}, $(R, \bar{\beta})$ is a 3-Lie-Rinehart algebra $(L\ltimes R, A, \rho_4)$-module, and by  \eqref{eq:barbeta},\eqref{eq:semider} and \eqref{eq:semiproductR},
    $\forall x, x_1, x_2, x_3\in L$, $r, r_1, r_2, r_3\in R$,
 \begin{equation*}
\begin{split}
  &\psi(x, r)[x_1+r_1, x_2+r_2, x_3+r_3]\\
  =&\psi(x, r)([x_1,x_2,x_3]+\beta(x_1,x_2)r_3+\beta(x_2,x_3)r_1+\beta(x_3,x_1)r_2\big)\\
  =&\beta([x_1,x_2,x_3],x)r,\\\\
&[\psi(x, r)(x_1+r_1), \psi(x, r)(x_2+r_2), \psi(x, r)(x_3+r_3)]+\bar{\beta}(x_1+r_1, x_2+r_2)\psi(x, r)(x_3+r_3)\\
+&\bar{\beta}(x_2+r_2, x_3+r_3)\psi(x, r)(x_1+r_1)+\bar{\beta}(x_3+r_3, x_1+r_1)\psi(x, r)(x_2+r_2)\\
=&\beta(x_1, x_2)\beta(x_3,x)r+\beta(x_2, x_3)\beta(x_1, x)r+\beta(x_3, x_1)\beta(x_2, x)r\\
=&\beta([x_1,x_2,x_3],x)r.
 \end{split}
\end{equation*}

Therefore, for all $x\in L, r\in R$, $\psi(x, r)$ satisfies \eqref{eq:derivation}. The result follows.
\end{proof}

\begin{theorem}\label{thm:inducedder}
  Let $(L, A, \rho)$ and  $(L', A, \rho')$ be 3-Lie-Rinehart algebras and $(R,\beta)$ be an $(L, A, \rho)$-module. If $f: L'\rightarrow L$ is a 3-Lie-Rinehart  homomorphism,  then for any $\psi\in Der_{\beta}(L, R)$, we have
 $\psi'=\psi f\in Der_{\beta'}(L', R)$,  and $\psi'$ is called {\bf an $f$-$A$-derivation}, where $(R, \beta')$ is the $(L', A, \rho')$-module induced by $f$ in Theorem \ref{thm:induced}.
 \end{theorem}

\begin{proof} Thanks to Theorem \ref{thm:induced}
 $(R, \beta')$ is an $(L', A, \rho')$-module, where $\beta'$ is defined by \eqref{eq:inducedbeta}, and $\psi'=\psi f: L'\rightarrow R$ be an $A$-linear mapping.

 By \eqref{eq:derivation}, $\forall x', y', z'\in L'$,
 \begin{equation*}
\begin{split}
 &\psi'([x',y',z'])=\psi f([x',y',z'])=\psi([f(x'), f(y'), f(z')])\\
 =&[\psi(f(x')), \psi(f(y')), \psi(f(z'))]+\beta(f(x'),f(y'))\psi(f(z'))\\
 +&\beta(f(y'),f(z'))\psi(f(x'))+\beta(f(z'),f(x'))\psi(f(y'))\\
 =&[\psi'(x'), \psi'(y'), \psi'(z')]+\beta'(x',y')\psi'(z')+\beta'(y',z')\psi'(x')+\beta'(z',x')\psi'(y').
 \end{split}
\end{equation*}
It follows that $\psi'\in Der_{\beta'}(L', R)$.
\end{proof}

\begin{theorem}
  Let $(L, A, \rho)$ and $(L', A, \rho')$ be 3-Lie-Rinehart algebras,  and $(R,\beta)$ be an $(L, A, \rho)$-module.
  Then for any $3$-Lie-Rinehart  homomorphism $f: L'\rightarrow L$ and any 3-Lie-Rinehart derivation  $\psi'\in Der_{\beta'}(L', R)$, there exists an unique 3-Lie-Rinehart  homomorphism  $h: L'\rightarrow L\ltimes R$
  such that
  \begin{equation}\label{eq:piq}
  \pi\cdot h=f,\quad q\cdot h=\psi', ~~ \mbox{where}~~q: L\ltimes R \rightarrow R, q(x+r)=r, \forall x\in L, r\in R,
  \end{equation} and $ \pi$ is defined as Eq \eqref{eq:pii}.

  Conversely, any 3-Lie-Rinehart homomorphism  $h: L'\rightarrow L\ltimes R$ determines a 3-Lie-Rinehart   homomorphism $f: L'\rightarrow L$ and a 3-Lie-Rinehart derivation
$\psi'\in Der_{\beta'}(L', R)$ satisfying  Eq \eqref{eq:piq}, where $\beta'$ is defined as \eqref{eq:inducedbeta}.
\end{theorem}

\begin{proof} Let $f: L'\rightarrow L$ be  a 3-Lie-Rinehart homomorphism. Then by Theorem \ref{thm:induced},
$(R, \beta')$ is a 3-Lie-Rinrhart algebra $(L', A, \rho')$-module induced by $f$. For any $\psi'\in Der_{\beta'}(L', R)$, defines  $ F$-linear mapping
\begin{equation}\label{eq:h}
h: L'\rightarrow L\ltimes R, ~~~h(x')=f(x')+\psi'(x'), ~~ \forall x'\in L'.
\end{equation}

 Thanks to Eq \eqref{eq:pii}, $ \forall x'\in L'$, $\pi\cdot h(x')=f(x')$ and $q\cdot h(x')=\psi'(x')$. Therefore, $h$ satisfies Eq \eqref{eq:piq}.

  Now we prove  that $h$ is a  3-Lie-Rinehart  homomorphism.

By Definition \ref{defin:homomorph}  and \eqref{eq:h} and \eqref{eq:derivation}, for $\forall a\in A$, $x', y', z'\in L'$,
   \begin{equation*}
   h(ax')=f(ax')+\psi'(ax')=a(f(x')+\psi'(x'))=a h(x'),
  \end{equation*}
\begin{equation*}
\begin{split}
&h([x',y',z'])=f([x',y',z'])+\psi'([x',y',z'])\\
=&[f(x'),f(y'),f(z')]+[\psi'(x'), \psi'(y'), \psi'(z')]+\beta'(x', y')\psi' (z')\\
&+\beta'(y',z')\psi' (x')+\beta'(z', x')\psi' (y')\\
=&[f(x'),f(y'),f(z')]+\beta(f(x'),f(y'))\psi' (z')\\
&+\beta(f(y'),f(z'))\psi' (x')+\beta(f(z'),f(x'))\psi' (y')\\
=&[f(x')+\psi'(x'), f(y')+\psi'(y'),f(z')+\psi' (z')]\\
=&[h(x'),h(y'),h(z')].\\
\end{split}
\end{equation*}
  Therefore, $h$ is a 3-Lie algebra homomorphism.

  Thanks to  Theorem \ref{thm:module},
$\forall x',y'\in L'$,
  \begin{equation*}
  \rho_4(h(x'),h(y'))=\rho_4(f(x')+\psi(x'),f(y')+\psi(y'))=\rho(f(x'),f(y'))=\rho'(x',y').
  \end{equation*}
  It follows that $h$ is a 3-Lie-Rinehart homomorphism.

    Conversely, if
    $h: L'\rightarrow L\ltimes R$  is a 3-Lie-Rinehart  homomorphism. Define $f: L'\rightarrow L$ and $\psi': L'\rightarrow R$ by
    $$ f(x')=\pi(h(x')), ~~ \psi'(x')=q(h(x')), ~~ \forall x'\in L'.
    $$
    Then  $ \forall x'\in L'$,   $f=\pi\cdot h$ and $ \psi'=q\cdot h$, that is, \eqref{eq:piq} holds.

    By  Theorem \ref{thm:module} and Eq \eqref{eq:semiproduct},  $ \forall x', y', z'\in L, a\in A,$ we have
                $\rho(f(x'),f(y'))$ $=\rho_4(h(x'),h(y'))$ $=\rho'(x',y'), f(ax')=af(x'),$ $ ~ \psi'(ax')=a\psi'(x'),$ and
     \begin{equation*}
\begin{split}
&h([x',y',z'])=[h(x'),h(y'),h(z')]\\
=&[f(x')+\psi'(x'),f(y')+\psi'(y'),f(z')+\psi'(z')]\\
=&[f(x'),f(y'),f(z')]+\beta(f(x'),f(y'))\psi'(z')+\beta(f(y'),f(z'))\psi'(x')\\
+&\beta(f(z'),f(x'))\psi'(y')\\
=&[f(x'),f(y'),f(z')]+\beta'(x',y')\psi'(z')+\beta'(y',z')\psi'(x')\\
+&\beta'(z',x')\psi'(y')=f([x',y',z'])+\psi'([x',y',z']).\\
\end{split}
\end{equation*}
Therefore,
 \begin{equation*}
f([x',y',z'])=[f(x'),f(y'),f(z')],
 \end{equation*}
 \begin{equation*}
 \begin{split}
\psi'([x',y',z'])=&\beta'(x',y')\psi'(z')+\beta'(y',z')\psi'(x')+\beta'(z',x')\psi'(y')\\
=& [\psi'(x'),\psi'(y'), \psi'(z')]+\beta'(x',y')\psi'(z')+\beta'(y',z')\psi'(x')+\beta'(z',x')\psi'(y').
\end{split}
 \end{equation*}

Therefore,  $f: L'\rightarrow L$ is a 3-Lie-Rinehart homomorphism and
$\psi'\in Der_{\beta'}(L', R)$.  The proof is complete.
\end{proof}

\section{Crossed module of 3-Lie-Rinehart algebras}

  In this section we discuss crossed modules of $3$-Lie-Rinehart algebras.

\begin{defn}\label{defin:cross}
  Let $(L,A,\rho)$ be a 3-Lie-Rinehart algebra, $(R,A)$ be a 3-Lie A-algebra,  and $\beta: L\wedge L\rightarrow Der(R)$ be an action of  3-Lie-Rinehart algebra   $(L, A, \rho)$ on  $(R, A)$. If a 3-Lie algebra homomorphism
  $\partial: R\rightarrow L$ satisfies:\\
  $(1)$ $\partial(\beta(x,y)r)=[x,y,\partial r], ~~ \forall r\in R,~ x, y\in L,$\\
  $(2)$ $\beta(\partial r_1,\partial r_2)r=[r_1,r_2,r], ~~ \forall r_1, r_2, r\in R,$\\
  $(3)$ $\beta(x,\partial r_1)r_2=-\beta(x,\partial r_2)r_1, ~~ \forall r_1, r_2\in R,$\\
  $(4)$ $\partial(ar)=a\partial(r), ~~ \forall a\in A,~ r\in R,$\\
  $(5)$ $\rho(\partial(r_1),\partial(r_2))(a)=0, ~~\forall a\in A, ~ r_1,r_2,r\in R,$\\
 then $(R, A, \beta, \partial)$ is called {\bf a crossed module of $3$-Lie-Rinehart algebra $(L,A,\rho)$}.
\end{defn}

\begin{theorem}\label{thm:inclusion} Let $(L, A, \rho)$ and $(L', A, \rho')$ be 3-Lie-Rinehart algebras, and $f: L\rightarrow L'$ be 3-Lie-Rinehart homomorphism,
$\partial: Ker(f)\hookrightarrow L$ be the including mapping, that is, $\partial(u)=u$, for all $u\in Ker(f)$. Then $(Ker(f), A, \mbox{ad}, \partial)$ is a crossed module of 3-Lie-Rinehart algebra $(L, A, \rho)$.
\end{theorem}

\begin{proof} By Proposition \ref{thm:prop2},
   $Ker(f)$ is an ideal of the 3-Lie-Rinehart algebra $(L, A, \rho)$. Therefore, $\mbox{ad}: L\wedge L\rightarrow Der(Ker(f))$ is an action of  $(L, A, \rho)$ on $(Ker(f), A)$.

   Thanks to Eq \eqref{eq:ad}, for all $x,y\in L$, $u_1, u_2, u\in Ker(f)$ and $a\in A$,

   $\partial(\mbox{ad}(x,y)u)=\mbox{ad}(x,y) u=[x,y, u]=[x,y,\partial u],$

  $\mbox{ad}(\partial u_1,\partial u_2)u=\mbox{ad}(u_1, u_2)u=[u_1,u_2,u],$~~  $\partial(au)=a u=a\partial(u),$

  $\mbox{ad}(x,\partial u_1)u_2=[x,\partial u_1,u_2]=[x,u_1,u_2]=-[x,u_2,u_1]=-[x,\partial u_2,u_1]=-ad(x,\partial u_2)u_1,$

   $\rho(\partial(u_1),\partial(u_2))(a)=\rho(u_1, u_2)a=\rho'(f(u_1), f(u_2))a=0,$
    the result follows.
 \end{proof}

\begin{coro}\label{coro:coross1}
 Let $(L,A,\rho)$ be a 3-Lie-Rinehart algebra and $N$ be a proper ideal of  $(L,A,\rho)$. Then $(N, A, \mbox{ad},i)$ is a crossed module of the $3$-Lie-Rinehart algebra $(L, A, \rho)$, where $i: N\rightarrow L$, $i(x)=x, \forall x\in N.$
\end{coro}

\begin{proof}
 Apply Theorem \ref{thm:inclusion}.
 \end{proof}

 \begin{coro}\label{coro:coross2}
  Let $(L,A,\rho)$ be a 3-Lie-Rinehart algebra,  and $(R,\beta)$ be an $(L, A, \rho)$-module. Then $(R, A, \beta,0)$ is a crossed module of  $(L, A, \rho)$.
 \end{coro}

\begin{proof}
 Apply Theorem \ref{thm:inclusion}.
 $\Box$ \end{proof}

\begin{theorem}
  Let $(L,A,\rho)$ be a 3-Lie-Rinehart algebra, $(R,A)$ be a 3-Lie A-algebra and $\partial: R\rightarrow L$ be a 3-Lie-Rinehart algebra  central epimorphism. Then $(R, A, \beta, \partial)$ is a crossed module of $(L,A,\rho)$, where $\beta: L\wedge L\rightarrow gl(R)$,
   \begin{equation}\label{eq:pa}
   \beta(x,y)r=[r_1,r_2,r], ~~\forall  x,y\in L, ~~r_1,r_2,r\in R, ~~\mbox{where}~~~\partial(r_1)=x, ~~ \partial(r_2)=y.
  \end{equation}

\end{theorem}

\begin{proof} Since $\partial: R\rightarrow L$ is a 3-Lie-Rinehart algebra  central epimorphism, by Definition \ref{defin:homomorph}, $\partial$ is a 3-Lie-Rinehart homomorphism satisfying $Ker(\partial)\subseteq Z(R)$ and  $\partial(R)=L$. Therefore, $\beta$ defined by Eq \eqref{eq:pa}  is valid, and   for $\forall a\in A, r, r_1, r_2\in R, $ we have $\partial(ar)=a\partial(r) $ and $\rho(\partial(r_1),\partial(r_2))(a)=0.$

Thanks to Eqs \eqref{eq:pa} and \eqref{eq:jacobi}, for  $\forall x,y\in L, r_i\in R, 1\leq r_i\leq 5,$ there exist $r_1, r_2\in R$ such that
  $~\partial(r_1)=x, ~~ \partial(r_2)=y$, and
  \begin{equation*}
\begin{split}
\beta(x,y)[r_3,r_4,r_5]=&[r_1,r_2,[r_3,r_4,r_5]]\\
=&[[r_1,r_2,r_3],r_4,r_5]+[r_3,[r_1,r_2,r_4],r_5]+[r_3,r_4,[r_1,r_2,r_5]]\\
=&[\beta(x,y)r_3,r_4,r_5]+[r_3,\beta(x,y)r_4,r_5]+[r_3,r_4,\beta(x,y)r_5],
\end{split}
\end{equation*}
$$\beta(x,\partial r_3)\partial r_4=[r_1,r_3,r_4]=-[r_1,r_4,r_3]=-\beta(x,\partial r_4)\partial r_3,$$
 $$\beta(\partial r_1,\partial r_2)r=\beta(x,y)r=[r_1,r_2,r],$$
  \begin{equation*}
  \partial(\beta(x,y)r)=\partial[r_1,r_2,r]=[\partial r_1,\partial r_2,\partial r]=[x,y,\partial r].
\end{equation*}

Thanks to Definition \ref{defin:cross}, $(R, A, \beta, \partial)$ is a crossed module of $(L,A,\rho)$.
\end{proof}

\begin{theorem} Let $(L,A,\rho)$ be a 3-Lie-Rinehrart algebra, $(R, A)$ be a 3-Lie $A$-algebra and $(R, A, \beta, \partial)$ be a crossed module of  $(L,A,\rho)$. Then we have\\
$1)$ $(Im(\partial), A, \rho|_{Im(\partial)\wedge Im(\partial)})$ is a Hypo-ideal of $(L, A, \rho)$.
\\ $2)$ $(Coker(\partial)=L/Im(\partial), A, \rho_0)$ is a 3-Lie $A$-algebra, where $\rho_0(Coker(\partial), Coker(\partial))=0$.\\
 $3)$ $(Ker(\partial), A)$ is an abelian ideal of $(R, A)$.\\
 $4)$ The action $\beta$ of $(L, A, \rho)$ on $(R, A)$ yields an action $\beta|_{Im(\partial)\wedge Im(\partial)}$ of  $(Im(\partial), A, \rho|_{Im(\partial)\wedge Im(\partial)})$ on $(Ker(\partial), A)$.
  \end{theorem}

\begin{proof}
Thanks to Definition \ref{defin:cross},  $Im(\partial)$ is a subalgebra of $L$, and for $\forall r_1,r_2,r\in R$, $x,y\in L$ and $a\in A$,
 $$a\partial(r)=\partial(ar)\in Im(\partial), ~~ [x,y,\partial r]=\partial(\beta(x,y)r)\in Im(\partial).$$ Therefore, $(Im(\partial), A, \rho|_{Im(\partial)\wedge Im(\partial)})$ is a Hypo-ideal of $(L, A, \rho)$. The 1) follows.

The result 2) follows from 1) directly.

Since $\partial: R\rightarrow L$ is a 3-Lie algebra homomorphism, $Ker(\partial)$ is an ideal of the 3-Lie algebra $R.$

Thanks to Definition \ref{defin:cross}, $\forall r\in R, r_1,r_2, r'\in Ker(\partial)$, $a\in A$,
$$\partial(ar')=a\partial(r')=0,
 [r_1,r_2,r]=\beta(\partial(r_1),\partial(r_2))r=0.
$$
Therefore, $Ker(\partial)$ is an $A$-module, and $[Ker(\partial), Ker(\partial), R]=0$. we get  3).

The result 4) follows from 1) and Definition \ref{defin:cross} directly.
\end{proof}

Let $(L, A, \rho)$ be a 3-Lie-Rinehart algebra satisfying \begin{equation}\label{eq:0rho}
  \rho(x,y)az=0, ~~ \forall x,y,z\in L, a\in A.
 \end{equation}
Thanks to Theorem \ref{thm:LL}, $(L\wedge L, \rho_2)$ is a Lie-Rinehart algebra, where $\rho_2: L\wedge L\rightarrow Der(R)$, $\rho_2(x\wedge y)=\rho(x, y)$, ~~ $\forall x, y\in L$.

Now let $(R=L\wedge L, A)$ be a Lie $A$-algebra with $\forall x\wedge y, x'\wedge y'\in R, a\in A$,
\begin{equation}\label{eq:0LL}[x\wedge y, x'\wedge y']= [x, y, x']\wedge y'+x'\wedge [x, y, y'], ~~ a(x\wedge y)=\frac{1}{2}(ax)\wedge y+\frac{1}{2}x\wedge (ay).
\end{equation}

From Theorem \ref{thm:do}, $( W(L, R, A), \rho_3)$ is a Lie-Rinehart algebra. And by Eqs \eqref{eq:do},   \eqref{eq:0rho} and $$\rho_2(x\wedge y)b(x'\wedge y')=\frac{1}{2}\big((\rho(x, y)b) x'\wedge y'+x'\wedge(\rho(x, y)b)y'\big)=0,$$
we have that  for
 $ \varphi\in Der(R), x,y\in L$,
 $ (\varphi,x\wedge y) \in W(L, R, A)$ if and only if
 \begin{equation}\label{eq:00rho}
\varphi(b(x'\wedge y'))=b\varphi(x'\wedge y'), ~~~~ \forall b\in A, x'\wedge y'\in R.
 \end{equation}

 Thanks to Lemma \ref{lem:barWar} and Eqs \eqref{eq:relation} and \eqref{eq:wmultiplication},  $ (\overline{W(L,R,A)}, \tilde{\rho_3})$ is a Lie-Rinehart algebra, where
 $\rho_3: \overline{W(L,R,A)} \rightarrow Der(A), \rho_3(\overline{(\varphi, x\wedge y)})=\rho(x, y)$, $\forall x, y\in L$.

By the above notations, we have the following result.

\begin{theorem}\label{thm:L} Let $(L, A, \rho)$ be a 3-Lie-Rinehart algebra satisfying Eq \eqref{eq:0rho}, and $R=L\wedge L$ be a Lie $A$-algebra satisfying Eq \eqref{eq:0LL} and $(L\wedge L, \rho_2)$ be a Lie-Rinehart algebra in Theorem \ref{thm:LL}.
  Then $(L\wedge L, A,\beta,\partial)$ is a crossed module of Lie-Rinehart algebra $(\overline{W(L,L,A)}, \tilde{\rho_3})$,
  where
   \begin{equation}\label{eq:db} 
   \beta: \overline{W(L,L,A)}\rightarrow Der(L\wedge L),~ \beta\overline{(\varphi,x\wedge y)}=\varphi,~~~ \forall \overline{(\varphi, x\wedge y)}\in \overline{W(L,L,A)}.
  \end{equation}
  \begin{equation}\label{eq:ad}
  \partial: L\wedge L\rightarrow \overline{W(L,L,A)}, ~~ \partial(x\wedge y)=\overline{(\mbox{Ad}(x\wedge y),0)}, ~~\forall x, y\in L,
   \end{equation}
    where $Ad(x\wedge y)\in Der(R), ~~ Ad(x\wedge y)(x'\wedge y')=[x\wedge y, x'\wedge y'], ~~ \forall x'\wedge y'\in R.$

  Moreover, $Ker(\partial)$ is the center of the Lie algebra $L\wedge L$, and $Ker(\partial)$ is an $A$-module.
\end{theorem}

\begin{proof}
  By  Theorem \ref{thm:LL} and Lemma \ref{lem:barWar}, $(L\wedge L, \rho_2)$ and  $(\overline{W(L,L,A)}, \tilde{\rho_3})$ are Lie-Rinehart algebras, and  $ \partial$ defined by \eqref{eq:ad} is an algebra homomorphism since for all $x, y, x', y'\in L$,
 \begin{equation*}
\begin{split}
   \partial([x\wedge y, x'\wedge y'])=&\partial([x, y, x']\wedge y')+\partial(x'\wedge [x, y, y'])\\
 =&\overline{(\mbox{Ad}([x, y, x'], y'),0)}+\overline{(\mbox{Ad}(x', [x, y, y']),0)}\\
  = =&\overline{(\mbox{Ad}([x, y, x'], y')+\mbox{Ad}(x', [x, y, y']),0)}\\
 =&\overline{([\mbox{Ad}(x,y),\mbox{Ad}(x',y')],0)}\\
 =&[\overline{(\mbox{Ad}(x, y),0)}, \overline{(\mbox{Ad}(x', y'),0)}]\\
 =&[\partial(x\wedge y), \partial(x'\wedge y')].
\end{split}
\end{equation*}

 By Eqs \eqref{eq:domul}, \eqref{eq:db},\eqref{eq:ad} and \eqref{eq:00rho}, $\forall b\in A, \overline{(\varphi,x\wedge y)}$ and $\overline{(\varphi', x'\wedge y')}\in  \overline{W(L,L,A)},$

 \vspace{3mm}$
\beta[ \overline{(\varphi,x\wedge y)}, \overline{(\varphi', x'\wedge y')}]=\beta(\overline{([\varphi, \varphi'], (x\wedge y)\wedge(x'\wedge y'))})$
\\$=[\varphi, \varphi']=[\beta(\overline{(\varphi,x\wedge y)}), \beta(\overline{(\varphi', x'\wedge y')})],
$

\vspace{2mm}$\beta( b\overline{(\varphi,x\wedge y)})=\beta( \overline{(b\varphi, b(x\wedge y))})=b\varphi=b\beta( \overline{(\varphi,x\wedge y)}),$

 \vspace{2mm}$
\beta(\overline{(\varphi,x\wedge y)})(b(x'\wedge y'))=\varphi(b(x'\wedge y'))=\varphi(b(x'\wedge y'))+\rho(x, y)b(x'\wedge y')$
\\$=\varphi(b(x'\wedge y'))+\rho_3(\overline{(\varphi,x\wedge y)})b(x'\wedge y').
$
\\Therefore, $\beta: \overline{W(L,L,A)}\rightarrow Der(L\wedge L)$ is  an action of  Lie-Rinehart algebra $(\overline{W(L,L,A)}, \tilde{\rho_3})$ on $(L\wedge L, \rho_2)$.

Thanks to Eqs \eqref{eq:db} and \eqref{eq:ad},  for all $\overline{(\varphi,x\wedge y)}\in  \overline{W(L,L,A)}$ and  $ x'\wedge y'\in L\wedge L,$

\vspace{3mm}
 $
 \partial(\beta\overline{(\varphi,x\wedge y)}(x'\wedge y'))=\partial(\varphi(x'\wedge y'))=\overline{(\mbox{Ad}(\varphi(x'\wedge y')),0)}
$

\vspace{2mm}$=\overline{([\varphi,\mbox{Ad}(x'\wedge y')],0)}=[\overline{(\varphi,x\wedge y)},\overline{(\mbox{Ad}(x'\wedge y'),0)}]$

\vspace{2mm}
$=
[\overline{(d,x\wedge y)},\partial(x'\wedge y')].$
\\Therefore,
\begin{equation*}
 \partial(\beta(\overline{(d,x\wedge y)})(x'\wedge y'))=[\overline{(d,x\wedge y)},\partial(x'\wedge y')].
\end{equation*}

  Thanks to Eqs \eqref{eq:db} and \eqref{eq:actionad},  for all $b\in A$, $x,y,x',y'\in L$,
\begin{equation*}
\beta(\partial(x\wedge y))(x'\wedge y')=\beta\overline{(\mbox{Ad}(x\wedge y),0)}(x'\wedge y')
=\mbox{Ad}(x\wedge y)(x'\wedge y')=[x\wedge y,x'\wedge y'],
\end{equation*}

\begin{equation*}
\partial(b\cdot x\wedge y)=\frac{1}{2}(\partial(bx\wedge y)+\partial(x\wedge by))=b\cdot\overline{(\mbox{Ad}(x\wedge y),0 )}=b\cdot \partial(x\wedge y),
\end{equation*}

$$
 \tilde{\rho_3}(\partial(x\wedge y))b=\tilde{\rho_3}\overline{(\mbox{Ad}(x\wedge y),0)}b=0.
$$
Therefore,  $(L\wedge L,A,\beta,\partial)$ is a crossed module of Lie-Rinehart algebra $(\overline{W(L,L,A)}, \tilde{\rho_3})$. Follows from  Eq \eqref{eq:ad}, $Ker(\partial)$ is the center of the Lie algebra $L\wedge L$, and
for $\forall b\in A, x\wedge y\in Ker(\partial) $, $\partial(b\cdot x\wedge y)=b\cdot \partial(x\wedge y)=0$. The proof is complete.
 \end{proof}

\section*{Acknowledgements}

The first author (R. Bai) was supported in part by the Natural
Science Foundation of Hebei Province (A2018201126).

\bibliography{}

\end{document}